\documentclass[10pt,a4paper]{article}
\usepackage{indentfirst}
\usepackage{amsfonts}
\usepackage{amsmath}
\usepackage{amsthm}
\usepackage{enumerate}
\usepackage{amssymb}
\usepackage{multirow}
\usepackage{graphicx}
\usepackage{xcolor}
\usepackage{tikz}
\usepackage{tabularx}
\usepackage{multicol}
\usepackage{subfig}
\usepackage{sidecap}
\usepackage{wrapfig}
\usepackage{float}
\usepackage{array}

\usepackage{tikz}
\tikzstyle{vertex}=[circle, fill=black, inner sep= 0, minimum size = 6]
\tikzstyle{edge}=[line width=1.2pt]
\tikzstyle{edge-dashed}=[line width=1.2pt, dashed]

\usepackage[top=3cm, bottom=3cm, left=2cm, right=2cm]{geometry} 
\newcommand{\Title}[1]{{\large\bf \uppercase{#1}}}
\newcommand{\Email}[1]{\\
\smallskip \small{e-mail:} \textnormal{#1}}
\newcommand{\Author}[2]{{\sc #1} \\
\medskip
{\small\it #2}}

\newtheorem{theorem}{Theorem}[section]
\newtheorem{corollary}[theorem]{Corollary}
\newtheorem{proposition}[theorem]{Proposition}
\newtheorem{conjecture}[theorem]{Conjecture}
\newtheorem{lemma}[theorem]{Lemma}

\def\qm{\rm{QM}}
\def\qmi{\chi^{QM}}
\def\nsd{\rm{NSD}}
\def\nsdi{\chi_{\sum}^e}
\def\qmnsd{\rm{QM\;NSD}}
\def\qmnsdi{\chi_{\sum}^{QM}}
\def\mnsdi{\chi_{\sum}^{M}}
\def\d{d}
\def\smc{\sigma_c}

\begin{document}

\begin{center}
\Title{Quasi-majority neighbor sum distinguishing edge-colorings}
\bigskip

\Author{Rafa{\l} Kalinowski$^1$, Monika Pil\'sniak$^1$, El\.zbieta Sidorowicz$^2$, El\.zbieta Turowska$^2$}
{$^1$Department of Discrete Mathematics, AGH University of Krakow, Krak\'ow, Poland\\
$^2$Institute of Mathematics, University of Zielona G\'{o}ra, Zielona G\'ora, Poland}
\Email{kalinows@agh.edu.pl, pilsniak@agh.edu.pl, e.sidorowicz@im.uz.zgora.pl, e.turowska@im.uz.zgora.pl}
\bigskip
\end{center}

\begin{abstract}
In this paper, a $k$-edge-coloring of $G$ is any mapping $c:E(G)\longrightarrow [k]$. 
The edge-coloring $c$ of $G$ naturally defines a vertex-coloring $\sigma_{c}: V(G) \to \mathbb{N}$, where $\sigma_{c}(v)=\sum_{u\in N_G(v)}c(vu)$ for every vertex $v\in V(G)$. The edge-coloring $c$ is said to be neighbor sum distinguishing if it results in a proper vertex-coloring $\sigma_{c}$, which that $\sigma_{c}(u) \neq \sigma_{c}(v)$ for every edge $uv$ in $G$.
 
We investigate neighbor sum distinguishing edge-colorings with local constraints, where the edge-coloring is quasi-majority at each vertex. Specifically, every vertex $v$ is incident to at most $\left\lceil\d(v)/2\right\rceil$ edges of one color. This type of coloring is referred to as quasi-majority neighbor sum distinguishing edge-coloring. The minimum number of colors required for a graph to have a quasi-majority neighbor sum distinguishing edge-coloring is called the quasi-majority neighbor sum distinguishing index. A graph is nice if it has no component isomorphic to $K_2$.  We prove that any nice graph admits a quasi-majority neighbor sum distinguishing edge-coloring using at most 12 colors. This bound can be improved for bipartite graphs and graphs with a maximum degree of at most 4. Specifically, we show that every nice bipartite graph can be colored with 6 colors, and every nice graph with a maximum degree of at most 4 can be colored with 7 colors. Additionally, we determine the exact value of the quasi-majority neighbor sum distinguishing index for complete graphs, complete bipartite graphs, and trees.

We also consider majority neighbor sum distinguishing edge-colorings, that is, when each vertex is incident to at most $d(v)/2$ edges with the same color.
 
\end{abstract}

\section{Introduction}

We focus on simple, finite graphs. The sets of vertices and edges of a graph $G$ are denoted by $V(G)$ and $E(G)$, respectively. The term {\it order} of $G$ refers to $|V(G)|$, while {\it size} of $G$ refers to $|E(G)|$. $N_{G}(v)$ ($N(v)$ for short) denotes the neighborhood of a vertex $v$ in a~graph $G$ and $\d_{G}(v)$ ($\d(v)$ for short) the degree of a vertex $v$ in a~graph $G$. For any set $S \subseteq V(G)$, the symbol $G[S]$ denotes the subgraph of $G$ induced by $S$. Let $G_1$, $G_2$ be graphs such that $V(G_1) \cap V(G_2)$ can be nonempty. By $G_1 \cup G_2$ we mean a graph with the vertex set $V(G_1) \cup V(G_2)$ and the edge set $E(G_1) \cup E(G_2)$.

In this paper, a {\it{$k$-edge-coloring}} of a~graph $G$ is any mapping $c:E(G) \rightarrow \left[k\right]$.  Any edge-coloring $c$ induces a~vertex-coloring $\smc:V(G) \rightarrow \mathbb{N}$ given by
\begin{eqnarray*}
\smc(v)=\sum_{u \in N(v)}c(vu),
\end{eqnarray*}
for every $v \in V(G)$. We say that a $k$-edge-coloring $c$ {\it{distinguishes}} vertices $u, v \in V(G)$ if $\smc(u)\neq\smc(v)$.  
A~$k$-edge-coloring of a~graph $G$~is termed {\it{neighbor sum distinguishing}} ($\nsd$ for short) if it distinguishes every pair of adjacent vertices, i.e. $\smc$ is a~proper vertex-coloring of $G$. The smallest $k$ for which there exists an $\nsd$ $k$-edge-coloring of a~graph $G$ is called the {\it neighbor sum distinguishing index} and is denoted by $\nsdi(G)$.

Observe that a graph $G$ always admits an edge-coloring that induces a~proper vertex-coloring, except when it includes $K_2$ as a~component. By assigning a~different power of 2 to each edge, we ensure that every vertex receives a~distinct sum of colors of incident edges. However, the two vertices of $K_2$ cannot be distinguished in this way. Therefore, we call $G$ {\it nice} whenever it lacks $K_2$ as a~component.

The concept of coloring edges so that it generates a proper vertex-coloring has been frequently considered, particularly with the emergence of the 1-2-3 Conjecture posed in 2004 by Karo\'nski, {\L}uczak and Thomason $\cite{KLT}$. To be specific, this hypothesis states that the smallest value of $k$ for which every nice graph $G$ has a~$k$-edge-coloring $c$   such that $\sigma_c(u)\neq\sigma_c(v)$ for every edge $uv \in E(G)$ is equal to 3. After a series of tens of papers, this conjecture was finally proven by Keusch $\cite{K2}$.

\begin{theorem} {\rm (Keusch \cite{K2})}
Every nice graph $G$ satisfies $\nsdi(G) \leq 3$.
\end{theorem}

When we assume the additional restriction that the $k$-edge-coloring must be proper, then we obtain another version of $\nsd$ edge-coloring which we call {\it{neighbor sum distinguishing proper edge-coloring}}. The smallest value of $k$ such that such a~coloring exists is denoted by $\chi'_{\sum}(G)$. This value is related to a conjecture proposed by Flandrin, Marczyk, Przyby{\l}o, Sacl\'e, and Wo\'zniak in \cite{FMPSW}. 

\begin{conjecture} {\rm (Flandrin et al. \cite{{FMPSW}})}
Every nice graph $G \neq C_5$ satisfies $\chi'_{\sum}(G)\le \Delta(G)+2$.
\end{conjecture}

This conjecture remains unresolved. Wang and Yan \cite{WaYa2014} showed that \linebreak $\chi'_{\sum}(G) \le \left\lceil (10\Delta(G)+2)/3 \right\rceil$ for graphs with $\Delta(G) \ge 18$. Additionally, Przyby{\l}o \cite{Pr19b} established that $\chi'_{\sum}(G) \le \Delta + O(\sqrt{\Delta})$, where $\Delta = \Delta(G)$.  Let ${\rm col}(G)$ be the \emph{coloring number} of $G$, defined as the smallest  $k$ such that there is  a vertex ordering of $G$ in which each vertex is preceded by at most $k-1$ of its neighbors. It is known that $\chi'_{\sum}(G) \le 2\Delta(G) + {\rm col}(G) - 1$ \cite{Pr15} and $\chi'{\sum}(G) \le \Delta(G) + 3{\rm col}(G) - 4$ \cite{PrzWo15}.

Dailly, Duch\'ene, Parreau, and Sidorowicz  in \cite{DDPS}  generalized these conjectures by defining the concept of \emph{neighbor sum distinguishing $d$-relaxed edge-coloring}. A $k$-edge-coloring is $d$-\emph{relaxed} if each monochromatic set of edges induces a subgraph with  maximum degree at most~$d$. If a $d$-relaxed $k$-edge-coloring is distinguishing, then it is called a \emph{neighbor sum  distinguishing $d$-relaxed $k$-edge-coloring}. The smallest  $k$ for which  there is a neighbor sum distinguishing $d$-relaxed $k$-edge-coloring of $G$ is denoted by $\chi'^d_{\sum}(G)$. Consequently, $\chi'^{1}_{\sum}(G) = \chi'_{\sum}(G)$, and $\chi'^{\Delta(G)}_{\sum}(G)$ correspond to the 1-2-3 Conjecture.

An approach to investigate graph parameters is to examine bounds on the maximum degree of a graph. While graphs with a maximum degree 2 (which are essentially forests of paths and cycles) are typically straightforward to analyze, the case of subcubic graphs (graphs with a maximum degree 3) often presents more complexity. For example, $\chi'_{\sum}(G)\le 6$ for every nice subcubic graph $G$ \cite{HuWaWa17} (the conjecture posits that the bound should be 5). In~\cite{DDPS}, the authors examined the 2-relaxed case for subcubic graphs and proved that $\chi'^{2}_{\sum}(G) \le 4$ for every nice graph $G$ with maximum degree at most 3.

The paper \cite{DS} by Dailly and Sidorowicz investigates edge-colorings that distinguish vertices and allow adjacent edges to have the same color, but with an additional restriction. It is required that the set of edges incident to a vertex is not monochromatic when the degree of the vertex is large enough. It is proven that every nice graph has a neighbor sum distinguishing $7$-edge-coloring such that the set of edges incident to a vertex of degree at least~6 is not monochromatic.

Inspired by these results, we explore an edge-coloring approach that allows each vertex $v$ to be incident with at most $\left\lceil {\d(v)}/{2}\right\rceil$ edges of one color.  We refer to an edge-coloring where every vertex $v$ has at most $\left\lceil {\d(v)}/{2}\right\rceil$ incident edges of one color as to a quasi-majority edge-coloring. Thus, the aim of our work is to merge two types of coloring, a quasi-majority edge-coloring and a neighbor sum distinguishing edge-coloring, resulting in a new variant called quasi-majority neighbor sum distinguishing edge-coloring.

Our paper is organized as follows. In Section~\ref{sec:quasi-majority}, we present the definition and some properties of quasi-majority edge-coloring. In Section \ref{sec:basic properties}, we define a quasi-majority neighbor sum distinguishing edge-coloring and outline its basic properties. Section \ref{sec:classes of graphs} focuses on special classes of graphs, where we determine the exact value of the quasi-majority neighbor sum distinguishing index for complete graphs, complete bipartite graphs and trees. Additionally, we prove that the quasi-majority neighbor sum distinguishing index of graphs with maximum degree 4 is at most 7. In Section \ref{sec:general upper bound}, we establish a constant upper bound on the quasi-majority neighbor sum distinguishing index of every nice graph. We demonstrate that every graph has a quasi-majority neighbor sum distinguishing 12-edge-coloring, and provide an upper bound on the quasi-majority neighbor sum distinguishing index in terms of the maximum degree of a graph. This result offers an upper bound better than 12 for graphs with small maximum degree. Bock, Kalinowski, Pardey, Pil\'sniak, Rautenbach, and Wo\'zniak \cite{BKPPRW} introduced a majority edge-coloring, where each vertex $v$ has at most ${\d(v)}/{2}$ edges incident in one color. By merging majority edge-coloring and neighbor sum distinguishing edge-coloring, we derive majority neighbor sum distinguishing edge-coloring. In Section \ref{sec:majority}, we discuss the impact of our results on majority neighbor sum distinguishing edge-coloring. In Section, \ref{sec:open problems} we present some open problems.


\section{Quasi-majority edge-coloring} \label{sec:quasi-majority}

A~$k$-edge-coloring $c$ of a~graph $G$~is called {\it quasi-majority at a vertex} $v\in V(G)$ if $v$ is incident to at most $\left\lceil {\d(v)}/{2}\right\rceil$ edges with color $\alpha$, for every color $\alpha\in\left[k\right]$. If $c$ is quasi-majority at every vertex of $G$, then it is called {\it quasi-majority} ($\qm$ for short).  The minimum value of $k$ for which there exists a~$\qm$ $k$-edge-coloring of a~graph $G$ is called {\it{quasi-majority index}} and is denoted by $\qmi(G)$.

An edge-coloring of a~graph $G$~is called {\it{majority}} if every vertex $v \in V(G)$ is incident to at most ${\d(v)}/{2}$ edges in one color. This concept was introduced in~~\cite{BKPPRW}, where the following theorems were proven.  
\begin{theorem}{\rm (Bock et al.  \cite{BKPPRW})}\label{m4}
Every finite graph of minimum degree at least $2$ admits a majority $4$-edge-coloring.
\end{theorem}
\begin{theorem} {\rm (Bock et al. \cite{BKPPRW})} \label{T2}
Let G be a~connected graph.
\begin{enumerate}
	\item If $G$ has an even number of edges or $G$ contains vertices of odd degrees, then $G$ has a~$2$-edge-coloring such that, for every vertex $v$ of $G$, at most $\left\lceil \frac{\d(v)}{2}\right\rceil$ of the edges incident to $v$ have the same color.
	\item If $G$ has an odd number of edges, all the vertices of $G$ have even degree, and $u$ is any vertex of $G$, then $G$ has a~$2$-edge-coloring such that, for every vertex $v$ of $G$ distinct from $u$, exactly $\frac{\d(v)}{2}$ of the edges incident to $v$ have the same color, and exactly $\frac{\d(u)}{2}+1$ of the edges incident to $u$ have the same color.
\end{enumerate}
\end{theorem}

From Theorem $\ref{T2}$ we immediately derive that if $G$ has an even size or $G$ contains vertices of odd degrees, then there is a~$\qm$ 2-edge-coloring of $G$. In turn, we see that if $G$ has an odd size, all vertices of $G$ have even degrees and $u$ is any vertex in $G$, then there is a~2-edge-coloring of $G$ such that at any vertex of $G$ distinct from $u$ this edge-coloring is $\qm$, while at $u$ exactly $\frac{\d(u)}{2}+1$ edges have the same color. It is enough to recolor one edge at $u$ with a third color and then the edge-coloring at $u$ is also $\qm$. From this we get the following important fact that we use in this paper.

\begin{corollary}\label{cor:QM_upper_bound}
Every graph $G$ satisfies $\qmi(G)\leq3$.
\end{corollary}

Another consequence of Theorem $\ref{T2}$ can be easily justified.

\begin{proposition}\label{thm:QM_equal_two}
$G$ has no quasi-majority $2$-edge-coloring if and only if $G$ has an odd number of edges and all vertices of $G$ have even degrees. 
\end{proposition}

Observe that there does not exist a~bipartite graph with an odd size where all the vertices have even degrees. Therefore, the following corollary is true.

\begin{corollary} \label{C1}
For every bipartite graph $G$ with $\Delta(G)\ge 2$ we have $\qmi(G)= 2$.
\end{corollary}


\section{Definition and basic properties} \label{sec:basic properties}

A~$k$-edge-coloring of a~graph $G$~is termed {\it{quasi-majority neighbor sum distinguishing}} ($\qmnsd$ for short) if it is quasi-majority and neighbor sum distinguishing. The minimum value of $k$~for which there exists a~$\qmnsd$ $k$-edge-coloring of a~graph $G$ is called the {\it{quasi-majority neighbor sum distinguishing index}} and is denoted by $\qmnsdi(G)$. Observe that only nice graphs admit a $\qmnsd$ edge-coloring, and there is no graph with the $\qmnsd$ index equal to 1, so $\qmnsdi(G)\ge 2$ for every nice graph $G$.

Let $G$ be a~nice graph. It is easy to see that
\begin{eqnarray*}
\nsdi(G)\leq \qmnsdi(G)  \leq \chi'_{\sum}(G). 
\end{eqnarray*} 

\noindent The inequality $\qmnsdi(G) \leq \chi'_{\sum}(G)$ is sharp, e.g. for every graph $G$ with $\Delta(G)=2$. This implies that if a nice graph $G$ is a~path or a cycle, then $\qmnsdi(G) = \chi'_{\sum}(G)$, so according to the propositions included in \cite{FMPSW}, the following two propositions are true. 

\begin{proposition} \label{prop:path}
We have $\qmnsdi(P_{3})=2$, and for every $n \geq 4$ we have $\qmnsdi(P_{n})=3$.
\end{proposition}

\begin{proposition} \label{prop:cycle}
We have $\qmnsdi(C_{5})=5$, and for every $n\geq3$ and $n\neq5$ we have  
$$\qmnsdi\left(C_{n}\right) = \left\{ \begin{array}{ll}
3, & \textrm{ if $n\equiv0\!\!\!\pmod{3}$,}\\
4, & \textrm{ otherwise.}
\end{array} \right.$$
\end{proposition}

Furthermore, in \cite{DDPS} subcubic graphs were considered and the following result has been proven. 

\begin{theorem}{\rm (Dailly et al. \cite{DDPS})}\label{thm:subcubic_relaxed}
		If $G$ is a~nice subcubic graph with no component isomorphic to $C_5$, then it admits an $\nsd$ $4$-edge-coloring such that every vertex of degree at least $2$ is incident to at least two edges of different colors.
	\end{theorem}

If every vertex of degree at least 2 in a subcubic graph is incident to at least two edges of different colors, then the edge-coloring is quasi-majority at every vertex and so the edge-coloring is quasi-majority. Thus, we obtain the following.

\begin{proposition}\label{thm:subcubic}
If $G$ is a~nice subcubic graph with no component isomorphic to $C_5$, then $\qmnsdi(G)\leq4$.
\end{proposition}

The upper bound in Theorem \ref{thm:subcubic} is sharp, since cycles $C_n$ for $n\equiv 1,2 \pmod{3}$ require 4 colors for a $\qmnsd$ edge-coloring.

\begin{proposition}\label{prop:qm two coloring}
   Let $G$ be a nice graph without two adjacent vertices of the same degree. If $\qmi(G)=2$, then $\qmnsdi(G)=2$.
\end{proposition}

\begin{proof}
   Let $c$ be a $\qm$ 2-edge-coloring of $G$. If a vertex $v$ has even degree $d(v) = 2k$, then $\sigma_c(v) = 3k$. If a vertex has odd degree $d(v) = 2k + 1$, then $\sigma_c(v)$ is either $3k + 1$ or $3k + 2$. Therefore, $\smc(v) = \smc(w)$ only if $d(v) = d(w)$. Since no two adjacent vertices have the same degree, the coloring $c$ distinguishes adjacent vertices.
\end{proof}

We can also use interval colorings to find a $\qmnsd$ edge-coloring. A $k$-edge-coloring of a graph $G$ is called an {\it interval coloring} if the colors of the edges incident to each vertex of $G$ are distinct and form an interval of consecutive integers.

\begin{proposition}
   Let $G$ be a nice graph without two adjacent vertices of the same degree. If $G$ has an interval coloring, then $\qmnsdi(G)=2$.
\end{proposition}

\begin{proof}
   Let $c'$ be an interval coloring of $G$. We cnstruct a new coloring $c$ as follows: $c(e) = 1$ if $c'(e) \equiv 1 \pmod{2}$, and $c(e) = 2$ if $c'(e) \equiv 0 \pmod{2}$ for $e \in E(G)$. Therefore, a vertex $v$ of even degree $d(v) = 2k$ is incident to $k$ edges of color 1 and $k$ edges of color 2, resulting in $\smc(v) = 3k$. If a vertex has odd degree $d(v) = 2k + 1$, it is incident to $k$ edges of color 1 and $k+1$ edges of color 2, or to $k+1$ edges of color 1 and $k$ edges of color 2, so $\smc(v)$ is either $3k+1$ or $3k+2$. Thus, the coloring is quasi-majority at every vertex, and $\smc(v) = \smc(w)$ only if $d(v) = d(w)$.
\end{proof}

 Every bipartite graph $G$ with $|V(G)| \leq 15$ admits an interval coloring, so every bipartite graph $G$ with $|V(G)|\leq 15$ in which there are no two adjacent vertices with the same degree has a quasi-majority neighbor sum distinguishing 2-edge-coloring.   


\section{Special classes of graphs} \label{sec:classes of graphs}

In this section we study the~$\qmnsd$ index of complete graphs, complete bipartite graphs, trees, and graphs with maximum degree at most $4$.

\subsection{Complete  graphs}
To determine the $\qmnsd$ index of complete graphs, we use the following two lemmas.

\begin{lemma}\label{lem:complete_odd}
    Every complete graph $K_{2k+1}$ has a $\qmnsd$ $3$-edge-coloring in which $k$ vertices are incident to  $k-1$ edges of color $2$ and $k+1$ vertices are incident to  $k$ edges of color $2$.
\end{lemma}

\begin{proof}
    We prove this by induction on the number of vertices. The lemma is true for $k=1$, that is, for the complete graph $K_3$. Assume that it is true for all complete graphs of odd order with fewer than $2k+1$ vertices. Let $V(K_{2k+1})=\{v_1,v_1,\ldots,v_{2k+1}\}$. We decompose $K_{2k+1}$ into two edge disjoint subgraphs $G_1$ and $G_2$ such that $K_{2k+1}=G_1\cup G_2$. Let $G_1=G[\{v_1,\ldots,v_{2k-1}\}]$ and $G_2$ be a~spanning subgraph of $G$ that contains edges $E(G_2)=\{v_{2k}v_{2k+1}\}\cup \{v_{2k}v_i:i\in \{1,\ldots,2k-1\}\}\cup \{v_{2k+1}v_i:i\in \{1,\ldots,2k-1\}\}$. The subgraph $G_1$ is isomorphic to $K_{2k-1}$, so by the induction hypothesis there is a~$\qmnsd$ 3-edge-coloring such that   $k-1$ vertices are incident to $k-2$ edges in color 2 and $k$ vertices are incident to  $k-1$ edges of color $2$. Let $c_1$ be such a~coloring and $v_1,\ldots, v_{k-1}$ be the vertices with   $k-2$ incident edges in color 2. Let $c_2$ be an edge-coloring of $G_2$ such that 
    \begin{itemize}
        \item $c_2(v_{2k}v_{2k+1})=2$;
        \item $c_2(v_{2k}v_{i})=2$ for $i\in \{1,\ldots,{k-1}\}$;
        \item $c_2(v_{2k+1}v_{i})=2$ for $i\in \{1,\ldots,{k-1}\}$;
        \item $c_2(v_{2k}v_{i})=1$ for $i\in \{k,\ldots,{2k-1}\}$;
        \item $c_2(v_{2k+1}v_{i})=3$ for $i\in \{k,\ldots,{2k-1}\}$.
    \end{itemize}
     
Then, let $c$ be the edge-coloring of $K_{2k+1}$ such that $c(e)=c_i(e)$ if $e\in E(G_i)$ for $i=1,2$. We claim that $c$ is the $\qmnsd$ 3-edge-coloring such that $k$ vertices are incident to $k-1$ edges in color 2 and $k+1$ vertices are incident to  $k$ edges of color $2$.

First, observe that $c$ is a~$\qm$ edge-coloring. By construction, $c_2$ is $\qm$ at $v_{2k}$ and $v_{2k+1}$. The coloring $c_1$ implies that if $v_i \in \{v_1, \ldots, v_{k-1}\}$, then it has $k-2$ edges colored with 2 in $G_1$. So, together with the two edges colored with 2 in $G_2$, every vertex $v_i$ for $i \in \{1, \ldots, k-1\}$ has  $k$ edges colored with 2. Since $c_1$ is a~$\qm$ edge-coloring, every $v_i$ for $i \in \{1, \ldots, k-1\}$ has at most $k$ edges colored with 1 and at most $k$ edges colored with 3. This implies that $c$ is $\qm$ at every vertex $v_i \in \{v_1, \ldots, v_{k-1}\}$. Since every vertex $v_j$, for $j \in \{k, \ldots, 2k-1\}$, is $\qm$ in $(G_1, c_1)$ and we use two different colors at each vertex in $c_2$, it follows that it is also $\qm$ in $(G, c)$. Thus, the edge-coloring $c$ is $\qm$.

Furthermore,   vertices $v_1,\ldots,v_{k-1},v_{2k},v_{2k+1}$ are incident to $k$ edges colored with 2.  Observe that every vertex $v_j$ for $j \in \{k, \ldots, 2k-1\}$ is adjacent to $k-1$ edges colored with 2, since edges incident to $v_j$ in $G_2$ have colors 1 and 3. 

Finally, we show that $c$ is an $\nsd$ edge-coloring. $\smc(v_{2k}) = \sigma_{c_2}(v_{2k}) = k + 2k = 3k$ and $\smc(v_{2k+1}) = \sigma_{c_2}(v_{2k+1}) = 2k + 3k = 5k$, so $\smc(v_{2k}) \neq \smc(v_{2k+1})$. For every vertex $v_i$ where $i \in \{1, \ldots, 2k-1\}$, $\smc(v_i) = \sigma_{c_1}(v_i) + 4$. Thus, since $c_1$ is an $\nsd$ edge-coloring, $\smc(v_i) \neq \smc(v_j)$ for $1 \leq i < j \leq 2k-1$. Observe that $3k - 3 = k - 1 + 2(k - 1) \leq \sigma_{c_1}(v_i) \leq 2(k - 1) + 3(k - 1) = 5k - 5$ for $i \in \{1, \ldots, 2k-1\}$. Thus, $\smc(v_{2k}) \neq \smc(v_i)$ and $\smc(v_{2k+1}) \neq \smc(v_i)$ for $i \in \{1, \ldots, 2k-1\}$. Therefore, $c$ is an $\nsd$ edge-coloring, which completes the proof.    
\end{proof}

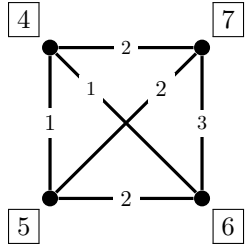
\begin{figure}[h]
\centering
\begin{tikzpicture}
\node[vertex] (0) at (0,0) {};
\node[vertex] (1) at (0,2) {};
\node[vertex] (2) at (2,0) {};
\node[vertex] (3) at (2,2) {};

\draw[edge] (0)to node[fill=white,midway,scale=0.85]{$1$} (1);
\draw[edge] (0)to node[fill=white,midway,scale=0.85]{$2$} (2);
\draw[edge] (0)to node[fill=white,near end,scale=0.85]{$2$} (3);
\draw[edge] (1)to node[fill=white,near start,scale=0.75]{$1$} (2);
\draw[edge] (1)to node[fill=white,midway,scale=0.75]{$2$} (3);
\draw[edge] (2)to node[fill=white,midway,scale=0.75]{$3$} (3);
    
\draw[edge] (0) node[below left] {\fbox{$5$}};
\draw[edge] (1) node[above left] {\fbox{$4$}};
\draw[edge] (2) node[below right] {\fbox{$6$}};
\draw[edge] (3) node[above right] {\fbox{$7$}};
\end{tikzpicture}
\caption{A $\qmnsd$ 3-edge-coloring of $K_4$}
\label{fig}
\end{figure}

\begin{lemma}\label{lem:complete_even}
Let $k\ge 2$. Every complete graph $K_{2k}$ has a $\qmnsd$ $3$-edge-coloring in which at least $k-1$ vertices are incident to at most $k-1$ edges of color $2$. 
\end{lemma}

\begin{proof}
    Similarly as for the odd case, the proof goes by induction on the number of vertices. The lemma is true for $k=2$, i.e. for the complete graph $K_4$ (see Fig.\ref{fig}). Assume that it is true for all complete graphs of even order with fewer than $2k$ vertices.  Let $V(K_{2k})=\{v_1,v_2,\ldots,v_{2k}\}$. We decompose $K_{2k}$ into two edge disjoint subgraphs $G_1$ and $G_2$ such that $G=G_1\cup G_2$. Let $G_1=G[\{v_1,\ldots,v_{2k-2}\}]$ and $G_2$ be a~spanning subgraph of $G$ that contains edges $E(G_2)=\{v_{2k-1}v_{2k}\}\cup \{v_{2k-1}v_i:i\in \{1,\ldots,2k-2\}\}\cup \{v_{2k}v_i:i\in \{1,\ldots,2k-2\}\}$. $G_1$ is isomorphic to $K_{2k-2}$, so by the induction hypothesis,  a~$\qmnsd$ 3-edge-coloring exists such that at least $k-2$ vertices are incident to at most $k-2$ edges in color 2. Let $c_1$ be such a~coloring and $v_1,\ldots, v_{k-2}$ be the vertices having at most $k-2$ edges in color 2. Let $c_2$ be an edge-coloring of $G_2$ such that 
    \begin{itemize}
        \item $c_2(v_{2k-1}v_{2k})=2$;
        \item $c_2(v_{2k-1}v_{i})=2$ for $i\in \{1,\ldots,k-2\}$;
        \item $c_2(v_{2k}v_{i})=2$ for $i\in \{1,\ldots,k-2\}$;
        \item $c_2(v_{2k-1}v_{i})=1$ for $i\in \{k-1,\ldots,2k-2\}$;
        \item $c_2(v_{2k}v_{i})=3$ for $i\in \{k-1,\ldots,2k-2\}$.
    \end{itemize}
     
Then, let $c$ be the edge-coloring of $K_{2k}$ such that $c(e)=c_i(e)$ if $e\in E(G_i)$, for $i=1,2$. 
We claim that $c$ is a~$\qm$ edge-coloring. By construction, $c_2$ is $\qm$ at $v_{2k-1}$ and $v_{2k}$. The coloring $c_1$ implies that if $v_i \in \{v_1, \ldots, v_{k-2}\}$, then it has at most $k-2$ edges in $G_1$ colored with 2. So, together with the two edges colored with 2 in $G_2$, every vertex $v_i$, for $i \in \{1, \ldots, k-2\}$, has at most $k$ edges colored with 2. Since $c_1$ is a~$\qm$ edge-coloring,  every vertex $v_i \in \{v_1, \ldots, v_{k-2}\}$ is incident to at most $k-1$ vertices in color 1 and at most $k-1$ vertices in color 3. Furthermore, since every vertex $v_j$, for $j \in \{k-1, \ldots, 2k-2\}$, is $\qm$ in $(G_1, c_1)$ and we use two different colors at each vertex in $c_2$, it follows that it is also $\qm$ in $(G, c)$. Thus, the edge-coloring $c$ is $\qm$.

Each vertex $v_j$, for $j \in \{k-1, \ldots, 2k-2\}$, is adjacent to at most $k-1$ edges colored with 2, since $c_1$ is a quasi-majority edge-coloring, and the edges incident to $v_j$ in $G_2$ are colored 1 and 3. Moreover, both $v_{2k-1}$ and $v_{2k}$ are adjacent to at most $k-1$ edges colored with 2. Therefore, there are at least $k+2$ vertices adjacent to at most $k-1$ edges of color 2. Hence, we can select $k-1$ vertices that satisfy the conditions of the lemma.

Finally, we show that $c$ is an $\nsd$ edge-coloring. We have $\smc(v_{2k-1}) = \sigma_{c_2}(v_{2k-1}) = 3k-2$ and $\smc(v_{2k}) = \sigma_{c_2}(v_{2k}) = 5k-2$, so $\smc(v_{2k-1}) \neq \sigma_{c}(v_{2k})$. For every vertex $v_i$ where $i \in \{1, \ldots, 2k-2\}$, $\smc(v_i) = \sigma_{c_1}(v_i) + 4$. Thus, since $c_1$ is an $\nsd$ edge-coloring, $\smc(v_i) \neq \smc(v_j)$ for $1 \leq i < j \leq 2k-2$. Observe that $3k - 5 \leq \sigma_{c_1}(v_i) \leq  5k - 7$ for $i \in \{1, \ldots, 2k-2\}$. Thus, $\smc(v_{2k}) \neq \smc(v_i)$ and $\smc(v_{2k+1}) \neq \smc(v_i)$ for $i \in \{1, \ldots, 2k-2\}$. Therefore, $c$ is an $\nsd$ edge-coloring.    
\end{proof}

\begin{theorem} \label{thm:complete}
   For $n\ge 3$, we have $\qmnsdi(K_{n})=3$. 
\end{theorem}

\begin{proof}
    From Lemmas \ref{lem:complete_odd} and \ref{lem:complete_even} it follows that $\qmnsdi(K_{n})\le 3$. Let $c$ be a~$\qm$ edge-coloring of $K_n$ with 2 colors. If $n$ is odd, then $\smc(v)=\frac{3n-3}{2}$ for every $v\in V(K_n)$, so $c$ is not $\nsd$. If $n$ is even, then $\smc(v)=\frac{3n}{2}-2$ or $\smc(v)=\frac{3n}{2}-1$ for every $v\in V(K_n)$, so again $c$ is not $\nsd$.
\end{proof}


\subsection{Bipartite graphs}

To prove an upper bound on the $\qmnsd$ index for bipartite graphs, we use the result established in \cite{KLT}.

\begin{theorem}{\rm (Karo\'nski et al. \cite{KLT})}
		\label{thm:group}
		Let $\Gamma$ be a~finite abelian group of odd order and let $G$ be a~non-trivial $|\Gamma|$-colorable graph. Then, there is an edge-coloring of $G$ with  elements of $\Gamma$ such that the resulting vertex-coloring is proper. 
	\end{theorem}

In \cite{DS}, it was noted that the proof of Theorem \ref{thm:group} leads to the following result.

\begin{theorem} {\rm (Dailly, Sidorowicz \cite{DS})} \label{T1}
Let $G$ be a~connected, nice bipartite graph with a bipartition $(V_1, V_2)$. Then $G$ admits an $\nsd$ $3$-edge-coloring. Moreover, there is an $\nsd$ $3$-edge-coloring $c$ of $G$ such that $\smc(v_1) \neq \smc(v_2) (\!\!\!\mod 3)$ for every $v_1 \in V_1$ and $v_2 \in V_2$. 
\end{theorem}

\begin{theorem} 
Every nice bipartite graph $G$ satisfies $\qmnsdi(G)\leq6$. 
\end{theorem}

\begin{proof}
Let $(V_1, V_2)$ be a bipartition of $G$. By Theorem $\ref{T1}$, there is an $\nsd$ 3-edge-coloring $c$ of the graph $G$ such that $\smc(v_1) \neq \smc(v_2) (\!\!\!\mod 3)$ for  $v_1 \in V_1,v_2 \in V_2$. By $E_i$ we denote the set of edges with color $i$ for $i\in [3]$. By Corollary $\ref{C1}$, there is $\qm$ 2-edge-coloring of every subgraph induced by $E_i$. Thus, we recolor some edges of $E_1$ with color 4 in such a~way that the coloring is $\qm$
at every vertex in subgraph induced by $E_1$. In a similar manner, we recolor certain edges of $E_2$ with the color 5 and some edges of $E_3$ with the color 6. Let $c'$ be the obtained edge-coloring. Thus $\smc(v) = \sigma_{c'} (\!\!\!\mod 3)$ and so $c'$ is an $\qmnsd$ 6-edge-coloring. 
\end{proof}

For complete bipartite graphs, we have a strict result. It is easily seen that $\qmnsdi\left(K_{2,2}\right) = 4.$
\begin{theorem}\label{compbip}
Every nice, complete bipartite graph $K_{n,m}$ such that $K_{n,m} \neq K_{2,2}$ satisfies
$$\qmnsdi\left(K_{n,m}\right) = \left\{ \begin{array}{ll}
2, & \textrm{ if $\;n\neq m$,}\\
3, & \textrm{ otherwise.}
\end{array} \right.$$
\end{theorem}

\begin{proof}

By Corollary \ref{C1}, $K_{n,m}$ has a $\qm$ 2-edge-coloring. If $n\neq m$, then no two adjacent vertices have the same degree. Thus, by Proposition \ref{prop:qm two coloring}, we have $\qmnsdi(K_{n,m})=2$ for $n\neq m$.

When $m=n$, two colors are not sufficient for a $\qmnsd$ edge-coloring of $K_{n,n}$. This is because the only possibility for a $\qm$ 2-edge-coloring of $K_{n,n}$ is to color $\left\lceil\frac{n}{2}\right\rceil$ edges at each vertex with one color and the remaining edges with the other color, but such an edge-coloring will not be $\nsd$. We now show that if $m=n$, then there exists a $\qmnsd$ 3-edge-coloring of $K_{n,n}$. 

Let $K_{n,n}=(V_1,V_2,E)$, where $V_1=\left\{a_1,\ldots,a_n\right\}$, $V_2=\left\{b_1,\ldots,b_n\right\}$. We decompose $K_{n,n}$ into the following three edge disjoint subgraphs $G_1$, $G_2$, and $G_3$.  Let $G_1=K_{n,n}\left[\left\{a_1,\ldots,a_{\left\lceil \frac{n}{2}\right\rceil},b_1,\ldots,b_{\left\lceil \frac{n}{2}\right\rceil}\right\}\right]$, $G_2=K_{n,n}\left[\left\{a_{\left\lceil \frac{n}{2}\right\rceil+1},\ldots,a_n,b_{\left\lceil \frac{n}{2}\right\rceil+1},\ldots,b_n\right\}\right]$ and let $G_3$ be a spanning subgraph of $K_{n,n}$ with  $E(G_3)=E(K_{n,n})\setminus(E(G_1)\cup E(G_2))$. Let $c_1$ be an edge-coloring of $G_1$ and $G_2$ such that every edge of $G_1$ and $G_2$ has color $1$. 

First, we consider the case when $n$ is odd. Let $c_2$ be an edge-coloring of $G_3$ such that $c_2(a_ib_j)=2$ for $i \le \left\lceil\frac{n}{2}\right\rceil$, $j \ge \left\lceil\frac{n}{2}\right\rceil+1$, and $c_2(a_ib_j)=3$ for $i \ge \left\lceil\frac{n}{2}\right\rceil+1$, $j \le \left\lceil\frac{n}{2}\right\rceil$. By the construction of $c_1$ and $c_2$,
\begin{itemize}
 \item at every $a_i$ with $i\leq \left\lceil \frac{n}{2}\right\rceil$, we have $\left\lceil \frac{n}{2}\right\rceil$ edges with color 1 and $\left\lfloor \frac{n}{2}\right\rfloor$ edges with color 2;
 \item at every $a_i$ with $i\ge \left\lceil \frac{n}{2}\right\rceil+1$, we have $\left\lceil \frac{n}{2}\right\rceil$ edges with color 3 and $\left\lfloor \frac{n}{2}\right\rfloor$ edges with 1;
\item at every $b_{i}$ with $i\leq \left\lceil \frac{n}{2}\right\rceil$, we have $\left\lceil \frac{n}{2}\right\rceil$ edges with color 1 and $\left\lfloor \frac{n}{2}\right\rfloor$ edges with 3;
\item at every $b_i$ with $i\ge \left\lceil \frac{n}{2}\right\rceil+1$,  we have $\left\lceil \frac{n}{2}\right\rceil$ edges with color 2 and $\left\lfloor \frac{n}{2}\right\rfloor$ edges with 1.    
\end{itemize}
The edge-coloring $c$ is $\qm$. We can also observe that $\smc(a_i)=\frac{3n-1}{2}$ for every vertex $a_i$ with $i\le\left\lceil \frac{n}{2}\right\rceil$; $\smc(a_i)=2n+1$ for every vertex $a_i$ with $i \ge \lceil \frac{n}{2}\rceil+1$; $\smc(b_i)=2n-1$ for every vertex $b_{i}$ with $i\le \lceil \frac{n}{2}\rceil$; and $\smc(b_i)=\frac{3n+1}{2}$ for every vertex $b_i$ with $i\ge \lceil \frac{n}{2}\rceil+1$. Hence, the edge-coloring $c$ is $\nsd$. 

Let now $n$ be even and let $c_2$ be an edge-coloring of $G_3$ such that $c_2(a_ib_j)=2$ if $i$ is odd, $j \in [n]$ and $c_2(a_ib_j)=3$ if $i$ is even, $j \in [n]$. By the construction of $c_1$ and $c_2$, we obtain: 
\begin{itemize}
    \item at every vertex $a_i$ with odd $i$, we have $\frac{n}{2}$ edges with color 1 and $\frac{n}{2}$ edges with color 2;
    \item at every $a_i$ with even $i$, we have $\frac{n}{2}$ edges with color 1 and $\frac{n}{2}$ edges with color 3;
    \item at every vertex $b_{i}$ with $i\le \frac{n}{2}$, we have $\frac{n}{2}$ edges with color $1$, $\left\lceil \frac{n}{4}\right\rceil$ edges with color 3, and $\left\lfloor \frac{n}{4}\right\rfloor$ edges with color 2;
    \item at every $b_i$ with $i\ge  \frac{n}{2} +1$, we have $\frac{n}{2}$ edges with color 1, $\left\lceil \frac{n}{4}\right\rceil$ edges with color 2, and $\left\lfloor \frac{n}{4}\right\rfloor$ edges with color 3.
\end{itemize}
It is easy to see that the edge-coloring $c$ is $\qm$. For every vertex $a_i$ with odd $i$, we have $\smc(a_i)=\frac{3n}{2}$; for every vertex $a_i$ with even $i$, we have $\smc(a_i)=2n$. If $n\equiv 0 \pmod{4}$, then $\smc(b_i)=\frac{7n}{4}$ for every vertex $b_i,\;i=\{1,\ldots,\frac{n}{2}\}$.  If $n\equiv 2 \pmod{4}$, then $\smc(b_i)=\frac{7n+2}{4}$ for  $i=\{1,\ldots,\frac{n}{2}\}$ and $\smc(b_i)=\frac{7n-2}{4}$ for  $i=\{\frac{n}{2}+1,\ldots,n\}$. Hence, the edge-coloring $c$ is also $\nsd$. 
\end{proof}


\subsection{Trees}

\begin{theorem} \label{thm:trees}
For each tree $T$ of order $n \geq 3$ we have 
$$\qmnsdi(T) = \left\{ \begin{array}{ll}
2, & \textrm{ if $\;T$ has no two adjacent vertices of equal even degree,}\\
3, & \textrm{ otherwise.}
\end{array} \right.$$
\end{theorem}

\begin{proof} Clearly, $\qmnsdi(T)\geq 2$ for every tree with at least three vertices. 
Furthermore, if in $T$ there are two adjacent vertices $x,y$ such that $\d(x)=\d(y)=k$ and $k$ is even, then to obtain a~$\qm$ edge-coloring, we have only two possibilities. We can color $xy$ with color 1 and assign color 2 to each of the $\frac{k}{2}$ edges incident to each vertex, with the remaining $\frac{k}{2}-1$ edges receiving color 1. Alternatively, we can color $xy$ with color 2 and assign color 1 to each of the $\frac{k}{2}$ edges incident to each vertex, with the remaining $\frac{k}{2}-1$ edges receiving color 2. In both cases, $\smc(x)=\smc(y)$; thus, we have to use three colors.

Pick a vertex $v_0$ as a root of the tree $T$. Let $l(v)$ be the distance of a vertex $v$ to the root $v_0$, and let $h(T)=\max_{v \in V(T)}l(v)$. Each value of $l(v)$ we call a {\it level} of $T$.

We show that there exists a~$\qmnsd$ 2-edge-coloring of $T$ if $T$ has no two adjacent vertices of equal even degree, and a~$\qmnsd$ 3-edge-coloring, otherwise. In both cases, we start the edge-coloring from the edges incident to the root $v_0$ and color these edges such that the edge-coloring at $v_0$ is $\qm$. Then we 
consider consecutive vertices in the BFS ordering until we color all edges in $T$. 

Suppose that $y$ is the first vertex in the BFS coloring with an uncolored edge. Let $x$ be its parent. Denote by $c$ the partial coloring obtained so far. Then $c$ is QM at $x$. We show how to extend the edge-coloring $c$ to uncolored edges incident to $y$ such that $\smc(x)\neq\smc(y)$ and $c$ is $\qm$ at $y$. Observe that if $y$ is a~leaf, then $\smc(x)\neq\smc(y)$, and $c$ already is a $\qm$ edge-coloring at $y$, thus we can assume that $y$ is a~non-leaf child of~$x$. 

\medskip
\noindent \textbf{Case 1.} $T$ has no two adjacent vertices of equal even degree.  

We prove that we can always color edges incident to $y$ with two colors such that the edge-coloring is $\qm$ at $y$ and distinguishes  $x$ and $y$ by sums. At the beginning note that if we want to color the edges incident to any vertex $v \in V(T)$ with a $\qm$ 2-edge-coloring, we have the following three types of possible colorings.

\begin{description}
    \item [(1)] $\d(v)$ is odd, $\frac{\d(v)+1}{2}$ edges are colored with 1, and $\frac{\d(v)-1}{2}$ edges with 2.
	\item [(2)] $\d(v)$ is odd, $\frac{\d(v)+1}{2}$ edges are colored with 2, and $\frac{\d(v)-1}{2}$ edges with 1.
    \item [(3)] $\d(v)$ is even, $\frac{\d(v)}{2}$ edges are colored with 1, and $\frac{\d(v)}{2}$ edges with 2.
\end{description}

\noindent If the edge-coloring is of type $(i)$ at $v$, then $\smc(v) \equiv i \pmod 3$, for $i\in\{1,2\}$. If the edge-coloring is type $(3)$ at $v$, then $\smc(v) \equiv 0 \pmod 3$.

Observe that if $\d(y)$ is even and since $\d(y) \geq 2$, then it follows that irrespective of the color of  $xy$, we have the opportunity to color the uncolored edges at $y$ so that we achieve at $y$ an~edge-coloring of type (3). On the other hand,  when $\d(y)$ is odd and since $\d(y) \geq 3$, regardless of a color of the edge $xy$, we can extend the edge-coloring at $y$ to an edge-coloring of type 1 and type 2.

We color the edges incident to $y$ in the following way.

\begin{itemize}
	\item If $\d(y)$ is even, then we apply such an edge-coloring to ensure that  we have a~coloring of type $(3)$ at $y$.
	\item If $\d(y)$ is odd, then the coloring of the edges incident to $y$ depends on the coloring of edges incident to $x$. When $\d(x)$ is even or $\d(x)$ is odd and the edge-coloring at $x$ is of type $(2)$, then we color the edges so that we obtain at $y$ the edge-coloring of type $(1)$. When $\d(x)$ is odd and the edge-coloring at $x$ is of type $(1)$, then we color the edges so that we obtain at $y$ the edge-coloring of type $(2)$.   
\end{itemize}

Obviously, the edge-coloring is $\qm$ at $y$. We claim that it distinguishes $x$ and $y$. If $\d(x)$ is odd and $\d(y)$ is even, then $\smc(x)\equiv 1\; \hbox{or}\;2 \pmod 3$ and $\smc(y)\equiv 0 \pmod 3$, therefore $\smc(x)\neq\smc(y)$. If $\d(x)$ and $\d(y)$ are both odd, then we have chosen the type of edge-coloring for $y$ such that $\smc(x)\not\equiv\smc(y) \pmod 3$. If $\d(x)$ is even and $\d(y)$ is odd, then $\smc(x)\equiv 0 \pmod 3$ and $\smc(y)\equiv 1\; \hbox{or}\;2 \pmod 3$, hence $\smc(x)\not\neq \smc(y)$. If both $\d(y),\d(x)$ are even, then $\smc(y) \neq \smc(x)$, since $\d(y)\neq\d(x)$. 

\medskip
\noindent \textbf{Case 2.} $T$ has two adjacent vertices of equal even degree.                      

Now we prove that we can always color the edges incident to $y$ with three colors so that the edge-coloring is $\qm$ at $y$ and distinguishes $x$ and $y$ by sums. First, observe that, in contrast to Case 1, there are many possibilities for coloring edges incident to a~vertex that result in a~$\qm$ edge-coloring, as we can use three colors. We define specific types of edge-colorings at the vertex, distinguishing between vertices of even degree and vertices of odd degree. We choose colors for edges in such a~way that each type results in a~sum at the vertex that is different modulo 3.

Let $v\in V(T)$ be a~vertex of even degree. We use the following three types of coloring of the edges incident to $v$.

\begin{description}
  \item [E1:]  $\frac{\d(v)}{2}$ edges have color 1, $\frac{\d(v)}{2}-1$ edges have color 2, and one edge has color 3; so $\smc(v) \equiv 1 \pmod 3$;
	\item  [E2:] $\frac{\d(v)}{2}-1$ edges have color with 1, $\frac{\d(v)}{2}$ edges have  color 2, and one edge  has color 3; so  $\smc(v) \equiv 2 \pmod 3$;
	\item  [E3:]  $\frac{\d(v)}{2}$ edges have color 1 and $\frac{\d(v)}{2}$ edges have color 2; so  $\smc(v) \equiv 0 \pmod 3$.
\end{description}
	
\noindent If $v\in V(T)$ has odd degree, then we use the following three types of coloring of the edges incident to $v$.

\begin{description}	
\item [O1:] $\frac{\d(v)-1}{2}$ edges we have color 1, $\frac{\d(v)-3}{2}$ edges have color 2, and two edges have color 3; so $\smc(v) \equiv 1 \pmod 3$;
\item [O2:] $\frac{\d(v)-3}{2}$ edges we have color 1, $\frac{\d(v)-1}{2}$ edges have color 2, and two edges have color 3; so $\smc(v) \equiv 2 \pmod 3$;
\item [O3:] if  $\frac{\d(v)-1}{2}$ edges we color with 1, $\frac{\d(v)-1}{2}$ edges have color 2 and one edge has color 3; so$\smc(v) \equiv 0 \pmod 3$.
\end{description}
	
\noindent Observe that depending on the color of the edge $xy$ and the parity of $d(y)$, one type of edge-coloring is available for the edges incident to $y$ so that  $\smc(x)\not\equiv \smc(y) \pmod 3$. Then, the edge-coloring distinguishes vertices $x$ and $y$, and is $\qm$ at $y$. 

Table \ref{tab} shows the types of coloring to which we extend the coloring $c$ to the edges incident to $y$ depending on $c(xy)$ and the type of coloring at $x$. One cell in the table is empty because in the coloring of type E3 there is no edge with color 3.

\begin{table}[h] 
\begin{center}
\begin{tabular}{|c|c|c|c|c|c|c|c|} 
  \cline{3-8}
  \multicolumn{2}{c|}{} & \multicolumn{6}{c|}{The type of coloring at $x$} \\
  \cline{3-8}
  \multicolumn{2}{c|}{} & E1 & E2 & E3 & O1 & O2 & O3 \\ 
  \hline
  \multirow{3}{*}{\!\!\!$c(xy)$}\!\!\!& 1 & E3, O3 &\!\!\!E1/E3, O1/O3\!\!\!&\!\!\!E1, O1\!\!\!& E3, O3 &\!\!\!E1/E3, O1/O3\!\!\!& E1, O1 \\
  \cline{2-8}
  & 2 &\!\!\!E2/E3, O2/O3\!\!\!& E3, O3 &\!\!\!E2, O2\!\!\!&\!\!\!E2/E3, O2\!\!\!& E3, O3 & E2, O2 \\
  \cline{2-8}
  & 3 & E2, O2/O3 & E1, O1/O3 &  &\!\!\!E2, O2/O3\!\!\!& E1, O1/O3 &\!\!\!E1/E2, O1/O2\!\!\!\\
  \hline
\end{tabular}
\end{center}
\caption{The types of colorings to which the coloring at $y$ is extended.}
\label{tab}
\end{table}
\end{proof}


\subsection{Graphs with maximum degree at most 4}

If $G$ is connected and $\Delta(G)\le 2$, then $G$ is either a cycle or a path. Thus, Propositions \ref{prop:path} and \ref{prop:cycle} give the value of $\qmnsd$ index for graphs with maximum degree at most 2. If $\Delta(G)\le 3$ and $G$ has no component isomorphic to $C_5$, then the $\qmnsd$ index is at most 4, by Theorem \ref{thm:subcubic}. In this subsection, we prove that there is a~$\qmnsd$ $7$-edge-coloring of every nice graph with maximum degree at most $4$.  The proof is based on the Combinatorial Nullstellensatz of Alon \cite{Al99}.

\begin{theorem}[Alon \cite{Al99}] \label{alon}
		Let $\mathbb{F}$ be an arbitrary field, and let $P = P(x_1,\ldots, x_n)$ be a polynomial in $\mathbb{F}[x_1,\ldots, x_n]$. Suppose the degree of $P$ equals $\sum_{i=1}^n k_i$, where each $k_i$ is a nonnegative integer, and suppose the coefficient of $x_1^{k_1}\cdots x_n^{k_n}$ in $P$ is nonzero. Then if $S_1,\ldots , S_n$ are subsets of $\mathbb{F}$ with $|S_i| > k_i$, there are $s_1\in S_1,\ldots ,s_n\in S_n$ such that $P(s_1,\ldots, s_n)\neq 0$.
	\end{theorem}

\begin{theorem}
    For every nice graph $G$ with maximum degree at most $4$, we have $$\qmnsdi(G) \le 7. $$
\end{theorem}

\begin{proof}
    We proceed by induction on the number of edges. By Theorem \ref{thm:subcubic}, the result is true for all nice subcubic graphs. Thus, the result is also true for graphs of size two or three. Assume that the result is true for graphs of size at most $m-1$.  Let $G$ be a nice graph with  $|E(G)|=m$ and $\Delta(G) = 4$. We may assume that $G$ is connected since otherwise, by induction, every component has a $\qmnsd$ 7-edge-coloring. Let $v\in V(G)$ and $d(v)=\Delta(G)=4$.

\medskip
\noindent {\bf Case 1. In $N(u)$, there are two adjacent vertices.}

Let $N(v)=\{v_1,v_2,v_3,v_4\}$ and $v_1v_2\in E(G)$. Let $G' = G - \{vv_1, vv_2\}$. Each component of $G'$ different from $K_2$ satisfies the claim. Observe that $G'$ has at most one component isomorphic to $K_2$.
By the induction hypothesis, there is a $\qmnsd$ 7-edge-coloring of the components of $G'$ that are different from $K_2$. Let $c$ be such a coloring. Additionally, if $G'$ contains a $K_2$, we extend $c$ by coloring it with any color from $[7]$.

To complete the edge-coloring, it remains to assign colors to the two edges $vv_1$ and $vv_2$ in such a way that all vertices in ${v, v_1, v_2}$ are distinguished from their neighbors. The coloring should be $\qm$ at $v,v_1,v_2,v_3,v_4$.

First, we count how many colors are forbidden for the edges $vv_1$ and $vv_2$ to obtain an edge-coloring ensuring that: (1) all neighbors, except for the pairs $(v_1,v_2)$, $(v,v_3)$, and $(v,v_4)$, are distinguished; and (2) the coloring is $\qm$ at each vertex, except $v$.

Consider the edge $vv_1$. Next, we analyze the situation based on the degree of $v_1$.  If $v_2$ is the only neighbor of $v_1$ in $G'$, then the color of $vv_1$ must be different from $c(v_1v_2)$, and therefore there is at most one forbidden color for $vv_1$. If $v_1$ has exactly two neighbors $v_2$ and $u_1$ in $G'$, then $v_1$ must be distinguished from $u_1$. However, regardless of a color we use for $vv_1$, the coloring will be $\qm$ at $v_1$ since $c(v_1u_1) \neq c(v_1v_2)$. Thus, again, at most one color is forbidden. If $v_1$ has three neighbors in $G'$, then at most one color must be forbidden to ensure that the coloring will be $\qm$ at $v_1$, and at most two additional colors must be forbidden to ensure that $v_1$ is distinguished from the two neighbors (except $v_2$). So, at most three colors are forbidden. In summary, the edge $vv_1$ can have at most three forbidden colors. Similarly, for the edge $vv_2$, there are also at most three forbidden colors.

We denote by $F_1$ the set of admissible colors for the edge $vv_1$, and by $F_2$ be the set of  admissible colors for the edge $vv_2$, so $|F_1| \geq 4$ and $|F_2| \geq 4$. By selecting a color for $vv_1$ from $F_1$ and a color for $vv_2$ from $F_2$, we obtain a 7-edge-coloring of $G$ in which all neighbors, except $(v_1, v_2)$, $(v, v_3)$, and $(v, v_4)$, are distinguished. Additionally, the coloring is $\qm$ at each vertex, except $v$.

Let $x_1 \in F_1$ and $x_2 \in F_2$ be the colors assigned to the edges $vv_1$ and $vv_2$, respectively. To ensure a $\qmnsd$ edge-coloring with the colors $x_1$ and $x_2$, the following conditions must be satisfied:

\vspace{-5pt}
\begin{itemize}\itemsep-3pt
\item $x_1 +x_2 + \smc(v) \neq \smc(v_i)$ -- we need to ensure that $v$ is distinguished from each $v_i$ for $i\in \{3,4\}$;
\item $x_2 + \smc(v) \neq \smc(v_1)$ -- we need to ensure that $v$ and $v_1$ are distinguished from each other;
\item $x_1 + \smc(v) \neq \smc(v_2)$ -- we need to ensure that $v$ and $v_2$ are distinguished from each other;
\item $x_1 + \smc(v_1) \neq x_2 + \smc(v_2)$ -- we need to ensure that $v_1$ and $v_2$ are distinguished from each other;
\item $x_1 \neq x_2$ -- we need to ensure that the coloring is $\qm$ at $v$.
\end{itemize}
\noindent  To demonstrate the existence of colors $x_1$ and $x_2$ that satisfy all the aforementioned conditions, we define the polynomial:

\vspace{-5pt}		
		\begin{align*}
			P(x_1,x_2) ={} & (x_1+x_2+\smc(v)- \smc(v_3)) \\
			& (x_1+x_2+\smc(v)- \smc(v_4) \\
			& (x_2+\smc(v)- \smc(v_1)) \\
			& (x_1+\smc(v)- \smc(v_2)) \\
			& (x_1-x_2+\smc(v_1)-\smc(v_2))\\
                & (x_1-x_2).
		\end{align*}

\noindent  
If there exist values of $x_1$ and $x_2$ such that $P(x_1, x_2) \neq 0$ and $x_i \in F_i$ for $i \in [2]$, then the $x_i$'s satisfy all the conditions. By coloring $vv_1$ and $vv_2$ with $x_1$ and $x_2$, respectively, we can extend the edge-coloring $c$ to a $\qmnsd$ edge-coloring.

We apply Theorem \ref{alon} to prove that $x_1$ and $x_2$ exist. First, we assert that the coefficient of the monomial $x_1^3 x_2^3$ is non-zero. Note that this coefficient in $P$ is identical to the one in the following polynomial:

	$$P_1(x_1,x_2)=(x_1+x_2)^2(x_1-x_2)^2x_1x_2.$$
 
\noindent The coefficient of the monomial $x_1^3 x_2^3$ is $-2$. Since $|F_1| > 3$ and $|F_2| > 3$, Theorem \ref{alon} implies that there are $x_1 \in F_1$ and $x_2 \in F_2$ such that $P(x_1, x_2) \neq 0$. Therefore, a $\qmnsd$ 7-edge-coloring of $G$ exists.

\medskip
\noindent {\bf Case 2. There is no edge in $N(v)$}
		
Let $N(v)=\{v_1,v_2, v_3,v_4\}$ and $G'=G-v$. Each component of $G'$ different from $K_2$ admits a $\qmnsd$  7-edge-coloring. Let $c$ be a $\qmnsd$  7-edge-coloring of components of $G'$ and additionally extend   $c$ to the components isomorphic to $K_2$, which we color with any color from $[7]$.

To complete the edge-coloring, we only need to color the edges $vv_i$ for $i \in [4]$. We select a color for each edge $vv_i$ in such a way that each $v_i$ is distinguished from its neighbors in $G'$, and the coloring is $\qm$ at each $v_i$ for $i \in [4]$. Additionally, after coloring the four edges $vv_1$, $vv_2$, $vv_3$, and $vv_4$, the vertex $v$ must be distinguished from its neighbors, and the coloring must be $\qm$ at $v$.

First, we count how many colors we need to forbid for the edges $vv_i$ to obtain an edge-coloring in which $v_i$ is distinguished from its neighbors in $G$, and the coloring is $\qm$ at $v_i$ for $i \in [4]$. We analyze the situation based on the degree of $v_i$ in $G'$.

If $d_{G'}(v_i)=0$, then we can use for $vv_i$ any color from $[7]$.

If $d_{G'}(v_i)=1$, then to distinguish $v_i$ from its neighbor,  at most one color is forbidden. Furthermore, the color of $vv_i$ has to be different from the color of the edge incident to $v_i$ in $G'$, resulting in at most two forbidden colors in total.

If $d_{G'}(v_i)=2$, then to distinguish $v_i$ from its neighbors,  at most two colors are forbidden. However, regardless of which color we use for $vv_i$, the coloring will still be $\qm$ at $v_i$, giving us at most two forbidden colors in total.

If $d_{G'}(v_i)=3$, then to distinguish $v_i$ from its neighbors,  at most three colors are forbidden. Furthermore, if $v_i$ is incident to two edges of one color, then this color is also forbidden for $vv_i$, so we have at most four forbidden colors in this case.

In summary, there are at most four forbidden colors for $vv_i$. We denote by $F_i$ the set of admissible colors for the edge $vv_i$, so $|F_i| \geq 3$ for $i \in [4]$, so $|F_i| \geq 3$ for $i \in [4]$. After coloring each edge $vv_i$ with a color $x_i \in F_i$ for $i \in [4]$, we obtain a 7-edge-coloring of $G$ in which all neighbours, except $(v, v_i)$ for $i \in [4]$, are distinguished. In addition, the coloring is $\qm$ at each vertex, except $u$.

Let $x_i \in F_i$ be the colors assigned to $vv_i$ for $i \in [4]$. To ensure a $\qmnsd$ edge-coloring,  for the colors $x_i$, the following conditions must also be met:
\vspace{-5pt}
		\begin{itemize}\itemsep-3pt
			\item $x_1+x_2+x_3+x_4-x_i\neq \smc(v_i)$ -- we need to ensure that  $v$ and $v_i$ are distinguished for $i\in [4]$;
			\item $x_1\neq x_2$ and $x_3\neq x_4$ -- we need to ensure that the coloring is $\qm$ at $v$.
		\end{itemize}
\noindent We consider the polynomial 
\vspace{-5pt}		
		\begin{align*}
			P(x_1,x_2,x_3,x_4)= & (x_2+x_3+x_4-\smc(v_1)) \\
			& (x_1+x_3+x_4- \smc(v_2)) \\
			& (x_1+x_2+x_4- \smc(v_3)) \\
			& (x_1+x_2+x_3- \smc(v_4)) \\
			& (x_1-x_2)(x_3-x_4).
		\end{align*}
\noindent If there are values $x_1, x_2, x_3, x_4$ such that $P(x_1, x_2, x_3, x_4) \neq 0$ and $x_i \in F_i$ for $i \in [4]$, then the values $x_i$ meet all the conditions. By coloring the edges $vv_i$ with the colors $x_i$, we can extend the edge-coloring $c$ to a $\qmnsd$ 7-edge-coloring.

We utilize Theorem \ref{alon}  again at this step to show that such $x_i$ exist. We examine the coefficient of the monomial $x_1^2 x_2 x_3^2 x_4$. Note that this coefficient in $P$ is identical to the one in the following polynomial:

$$P_1(x_1,x_2,x_3,x_4)=(x_2+x_3+x_4)(x_1+x_3+x_4)(x_1+x_2+x_4)(x_1+x_2+x_3)(x_1-x_2)(x_3-x_4).$$
\noindent The coefficient of the monomial  $x_1^2x_2x_3^2x_4$ is equal to 3. Since $|F_1|>2,|F_2|>1, |F_3|>2$ and $|F_4|\ge 1$,  Theorem \ref{alon} guarantees that there exist values   $x_i\in F_i,i\in [4]$ such that $P(x_1,x_2,x_3,x_4)\neq 0$ and therefore a $\qmnsd$ 7-edge-coloring of $G$ exists.
\end{proof}


\section{A~general upper bound} \label{sec:general upper bound} 

In \cite{KKP}, Kalkowski, Karo\'nski and Pfender proved that $\nsdi(G) \leq 5$ for every nice graph $G$. Using a~modification of their proof, we obtain the following upper bound for $\qmnsdi(G)$ of an arbitrary nice graph $G$.  

\begin{theorem}\label{genqm}
For every nice graph $G$,  $$\qmnsdi(G) \leq 12.$$
\end{theorem}

\begin{proof}
Without loss of generality, we may assume that $G$ is connected since otherwise we could consider connected components separately. Moreover, we may assume that $|V(G)|\geq 3$ and $\Delta(G)\geq 2$ since $G$ is a~nice graph, and there is nothing to prove for $G=K_1$. We present a~construction of a~$\qmnsd$ coloring $c:E(G)\rightarrow [12]$.

Order the vertex set $V(G)=\{v_1,\ldots, v_n\}$ such that $d(v_n)\geq 2$, and each vertex $v_i$, for $i<n$, has a~neighbor $v_j$ with $j>i$. Take any $\qm$ 3-edge-coloring $c_0$ of $G$ with colors $5,6,7$ and put an initial value $c(e)=c_0(e)$ for each edge $e\in E(G)$.  In our algorithm, we can change the color $c(e)$ of each edge $e$ at most twice. To every vertex $v_i$ with $i<n$, we will assign a~set $W(v_i)=\{w(v_i),w(v_i)+4\}$ of possible final values of $\sigma_c(v_i)$ such that $w(v_i)\in\{0,1,2,3\} \pmod 8$, and $W(v_i)\cap W(v_j)= \emptyset$ for every $v_iv_j\in E(G)$. In the final step of our algorithm, we will adjust the colors of the edges incident to $v_n$ to guarantee that $\sigma_c(v_n)\neq \sigma_c(v_i)$ whenever $v_iv_n\in E(G)$.

In the first step, we count $\smc(v_1)=\sum_{u \in N(v_1)}c_0(v_1u)$, and define the sets $W(v_1)=\{w(v_1),w(v_1)+4\}$ by putting $w(v_1)=\smc(v_1)$ if $\smc(v_1)\in \{0,1,2,3\}\pmod 8$, or $w(v_1)=\smc(v_1)-4$ otherwise. 

Let $2\leq k\leq n-1$, and assume that we have already established the set $W(v_i)$ for each $i<k$, and

\noindent
$(1)\quad c(v_iv_j)\in [12]$ for $1\leq i<j\leq n$,\\
$(2) \quad\smc(v_i)\in W(v_i)$ for $i<k$, \\
$(3)\quad c(v_kv_j)=c_0(v_jv_k)$ for $j>k$,\\
$(4)\quad  c(v_iv_k)\in \{5,6,7,8\}$ for $i<k$. 

In view of condition $(4)$, if $v_iv_k\in E(G)$, then we can either add or subtract 4 to $c(v_iv_k)$, so that the resulting coloring $c$ is still quasi-majority and $\smc(v_i)\in W(v_i)$. If $v_k$ has $d$ neighbors $v_i$ with $i<k$, then this gives us $d+1$ choices for $\smc(v_k)$. Additionally, we can change $c(v_kv_{j_0})$, where $j_0$ is the smallest $j>k$ with $v_kv_j\in E(G)$, using the following rule. It is easy to see that if we want to change the color of $c(v_kv_{j_0})$, then for each of its end-vertices $v_k,v_{j_0}$, there is at most one color violating the quasi-majority coloring of that vertex. Hence, we can choose a~new color $c(v_kv_{j_0})\in \{5,6,7,8\}$ such that $c$ is still a~quasi-majority edge-coloring of $G$. This gives us additional $d$ choices for $\smc(v_k)$. In total, we have $2d+1$ possible values for $\smc(v_k)$, at least one of them does not belong to any $W(v_i)$ for $i<k$. Therefore, a~possible recoloring of some edges $v_iv_k$ with $i<k$, and an edge $v_kv_{j_0}$, results in an edge-coloring $c$ satisfying the conditions:

\noindent
$(1')\quad c(v_iv_j)\in [12]$ for $1\leq i<j\leq n$,\\
$(2') \quad\smc(v_i)\in W(v_i)$ for $i\leq k$, \\
$(3')\quad c(v_kv_j)=c_0(v_jv_k)$ for $j>k$, except for $j=j_0$,\\
$(4')\quad  c(v_iv_n)\in \{5,6,7,8\}$ for $i<n$. 

This way, we successively assign pairwise disjoint sets $W(v_k)$ for all $k\leq n-1$.

In the final step, we have to determine $\smc(v_n)$. If $v_iv_n\in E(G)$ for some $i<n$, then condition $(4')$ allows us to subtract or add 4 to $c(v_iv_n)$ such that $\smc(v_i)\in W(v_i)$. Hence, we have $d(v_n)+1\geq 3$ possible options for $\smc(v_n)$. Let $s$ be the smallest such possible option, which we obtain by subtracting 4 from $c(v_i)$ for all $v_i\in N(v_n)$ with $\smc(v_i)=w(v_i)+4$, and adding 4 nowhere. If $s\in \{4,5,6,7\} \pmod 8$, then $s$ cannot be equal to $\smc(v_i)$ for any neighbor $v_i$ of $v_n$ since otherwise we could increase $s$ by subtracting 4 from $c(v_iv_n)$. Hence, we can put $\smc(v_n)=s$. Let then $s\in\{0,1,2,3\} \pmod 8$. If there exists a~$v_i\in N(v_n)$ with $\smc(v_i)\neq s$, then we keep $c(v_i)$ unchanged and subtract 4 everywhere else, thus obtaining a~suitable value $\smc(v_n)=s+4$. If $\smc(v_i)=w(v_i)$ for all $v_i\in N(v_n)$ then we subtract 4 from all edges incident to $v_n$ except two of them. 

The resulting values of $\smc$ yield a~proper vertex-coloring of $G$. 
\end{proof}

\medskip
For $\Delta(G)=5$ and $\Delta(G)=6$, we get a~better upper bound for $\qmnsdi(G)$ than in the results given previously.

\begin{theorem} 
For every nice graph $G$,
\begin{eqnarray*}
\qmnsdi(G) \leq \left\lceil \frac{3\Delta+4}{2}\right\rceil. 
\end{eqnarray*} 
\end{theorem}

\begin{proof}
We can assume that $G$ is connected, as otherwise, we could analyze its connected components independently. Furthermore, we may take $|V(G)| \geq 3$ and $\Delta(G) \geq 2$, given that $G$ is a nice graph, and the case $G=K_1$ does not require any further examination.

Let $m=|E(G)|$. We employ induction with respect to $m$. For $m=2$, the claim is obvious, so assume that $m\geq3$, and that a suitable edge-coloring exists for every graph with less than $m$ edges. Let $G$ be a graph with $m$ edges. Take any vertex $y$ of degree $\d(y)\geq 2$ and edges $xy$ and $yz$ incident to it. Assume that $x$ and $z$ are not adjacent; otherwise, $G$ is a complete graph, and the assertion follows directly from Theorem \ref{thm:complete}. Let $H$ denote the graph obtained by deleting the edges $xy$ and $yx$ from $G$. Under the induction hypothesis, we know that each component of $H$ with at least two edges can be colored with a $\qmnsd$ $\left\lceil \frac{3\Delta+4}{2}\right\rceil$-edge-coloring. For the rest of the components, which consist of isolated edges or vertices, we assign color $1$. We demonstrate that this coloring can be extended to the required edge-coloring of $G$.

Suppose that the neighborhood sets of the vertices $x$, $y$, and $z$ in graph $H$ are non-empty. Let $x_{1},\ldots,x_{p}$ be the vertices adjacent to $x$ in $H$, $y_{1},\ldots,y_{q}$ be the vertices adjacent to $y$ in $H$, and $z_{1},\ldots,z_{r}$ be the vertices adjacent to $z$ in $H$. Then $p,r$ do not exceed $\Delta-1$ and $q$ does not exceed $\Delta-2$. We begin by determining the number of colors available to color the edge $xy$ in order to extend edge-coloring to the desired one.
\begin{itemize}
	\item We cannot use one color to avoid an edge-coloring  not $\qm$ at $x$.
	\item We may have to exclude at most $\Delta-1$ colors because, while choosing the color of $xy$ we fix the sum at $x$, which may be the same as the sums already determined at  $x_{1},\ldots,x_{p}$.
	\item Moreover, one color can produce the same sums at $y$ and $z$ (regardless of the color chosen for the edge $yz$, which contributes to the sums at both $y$ and $z$).
\end{itemize}
In total, we have at least $\left\lceil \frac{3\Delta+4}{2}\right\rceil-1-(\Delta-1)-1=\left\lceil \frac{\Delta+2}{2}\right\rceil$ free choices of colors for $xy$. Analogously, we also have at least $\left\lceil \frac{\Delta+2}{2}\right\rceil$ free colors
for $xy$, which is not colored yet.

All we have to do now is to choose the colors for $xy$ and $yz$. From lists of lengths at least $\left\lceil \frac{\Delta+2}{2}\right\rceil$ we have to choose two colors so that they are different (otherwise the edge-coloring at $y$ will not be $\qm$) and such that the sum at $y$ is distinct from the at most $\Delta-2$ sums already determined at $y_{1},\ldots,y_{q}$. To demonstrate that this is always achievable, it is enough to observe that for any two lists of numbers $A$, $B$, each containing $s$ elements, there are at least $2s-3$ pairs $(a_{i},b_{i}) \in A \times B$ such that $a_{i} \neq b_{i}$, $i=1,\ldots,2s-3$, for which every sums $a_{i}+b_{i}$ are all different from each other. Specifically, for $s=\left\lceil \frac{\Delta+2}{2}\right\rceil$ we obtain $2s-3 \geq (\Delta-2)+1$. In fact, these are, for example, the pairs from $(\left\{a\right\} \times (B \setminus \left\{a\right\})) \cup ((A \setminus \left\{a, b\right\}) \times \left\{b\right\})$ , where $a= \min A$ and $b = \max B$.

The resulting edge-coloring meets our conditions, even if there are isolated edges or vertices within the components of $H$. Note also, if some of the the neighborhood sets of the vertices $x$, $y$, and $z$ in $H$ are empty, the proof is analogous.
\end{proof}


\section{Majority neighbor sum distinguishing edge-colorings} \label{sec:majority}

A~$k$-edge-coloring of a~graph $G$~is called {\it majority  neighbor sum distinguishing} if it is neighbor sum distinguishing and is majority. The minimum value of $k$~for which there exists a majority neighbor sum distinguishing $k$-edge-coloring is denoted by $\mnsdi(G)$. Recall that majority edge-colorings exist only for graphs with minimum degree at least 2.
Observe, when $\Delta(G) \leq 3$, then $\mnsdi(G)=\chi'_{\sum}(G)$, because the majority edge-coloring must then be proper.   If every vertex of the graph $G$ has even degree, then $\mnsdi(G)=\qmnsdi(G)$.  Thus, our results on $\qmnsdi(G)$ presented in the previous sections can be transferred to $\mnsdi(G)$.

First, we  determine the majority neighbor sum distinguishing index of complete graphs. We require the following lemma.

\begin{lemma} \label{lem:majority complete}
   If $k \ge 3$, there exists a majority $\nsd$ $4$-edge-coloring $c$ of $K_{2k}$ such that $v_{k}$ has at most $k-2$ edges of color $2$, or $v_{k+1}$ has at most $k-2$ edges of color $3$. Here $(v_1, \ldots, v_{2k})$ denotes the ordering of vertices of $K_{2k}$ with $\smc(v_i) < \smc(v_j)$ for $i < j$.
\end{lemma}

\begin{proof}
    We prove this by induction on the number of vertices. The lemma is true for $k = 3$, see the majority $\nsd$ 4-edge-coloring of $K_{6}$ in Figure \ref{fig:four}.

    Assume that this is true for all complete graphs of even order with fewer than $2k$ vertices.  We decompose $K_{2k}$ into two edge disjoint subgraphs $G_1$ and $G_2$ such that $G=G_1\cup G_2$. Let $G_1$ be a complete graph $K_{2k-2}$ and let $V(G_1)=\{v_1,\ldots,v_{2k-2}\}$. Let $x,y$ be remaining two vertices of $K_{2k}$ and $G_2$ be a~spanning subgraph of $K_{2k}$ that contains edges $E(G_2)=\{xy\}\cup \{xv_i:i\in \{1,\ldots,2k-2\}\}\cup \{yv_i:i\in \{1,\ldots,2k-2\}\}$. By the induction hypothesis $G_1$ has a~majority $\nsd$ 4-edge-coloring $c_1$   such that $\sigma_{c_1}(v_i) < \sigma_{c_1}(v_j)$ for $i < j$ and $v_{k-1}$ has at most $k-3$ edges of color 2, or $v_{k}$ has at most $k-3$ edges of color 3. The coloring of $G_2$ depends on whether $v_{k-1}$ has at most $k-3$ edges of color 2, or $v_k$ has at most $k-3$ edges of color 3. Therefore, we consider two cases.

    \medskip
    \noindent
    {\bf Case 1.} $v_{k-1}$ has at most $k-3$ edges of color 2 in the coloring $c_1$.
    
     Let $c_2$ be an edge-coloring of $G_2$ such that 
    \begin{itemize}
        \item $c_2(xy)=3$;
        \item $c_2(xv_{i})=1$ for $i\in \{1,\ldots,{k-2}\}$;
        \item $c_2(xv_{i})=2$ for $i\in \{k+1,\ldots,2k-2\}$;
        \item $c_2(yv_{i})=3$ for $i\in \{1,\ldots,{k-2}\}$;
        \item $c_2(yv_{i})=4$ for $i\in \{k+1,\ldots,{2k-2}\}$;
        \item $c_2(xv_{k-1})=2$, $c_2(yv_{k-1})=2$;
        \item $c_2(xv_{k})=1$, $c_2(yv_{k})=4$.
    \end{itemize}

    Let $c$ be the edge-coloring of $K_{2k}$ such that $c(e)=c_i(e)$ if $e\in E(G_i)$, for $i=1,2$. 
We claim that $c$ is a~majority edge-coloring. By construction, $c_2$ is majority at $x$ and $y$. The vertex $v_{k-1}$ has at most $k-3$ edges in $G_1$ colored with 2. So, together with the two edges in $G_2$ colored with 2, $v_{k-1}$ has at most $k-1$ edges colored with 2. Since $c_1$ is a~majority edge-coloring,  $v_{k-1}$ is incident to at most $k-2$ vertices in color 1 and at most $k-2$ vertices in color 3. Furthermore, since every vertex $v_j$, for $j \in \{1,\ldots, k-2,k,k+1 \ldots, 2k-2\}$, is majority in $(G_1, c_1)$ and is majority in $(G_2, c_2)$, it follows that it is also majority in $(G, c)$. Thus, the edge-coloring $c$ is majority.

We show that $c$ is an $\nsd$ edge-coloring. For each vertex $v_i$, where $i \in \{1, \ldots, k-1\}$, we have $\smc(v_i) = \sigma_{c_1}(v_i) + 4$, and for each vertex $v_i$ where $i \in \{k+1, \ldots, 2k-2\}$, we have $\smc(v_i) = \sigma_{c_1}(v_i) + 6$. Additionally, $\smc(v_k) = \sigma_{c_1}(v_k) + 5$. Thus, the coloring distinguishes the vertices $\{v_1, \ldots, v_{2k-2}\}$, with $\smc(v_i) < \smc(v_j)$ for $i < j$. As $\smc(x) = 3k$ and $\smc(v_1) \ge 3(k-1) + 4$, we have $\smc(x) < \smc(v_1)$. Similarly, since $\smc(y) = 7k - 5$ and $\smc(v_{2k-2}) \ge 7(k-1) - 5 + 6$, we have $\smc(y) > \smc(v_{2k-2})$. Therefore, $c$ is indeed an $\nsd$ edge-coloring.

Since $c_1$ is a majority coloring, $v_k$ is incident to at most $k-2$ edges of color 3. In $c_2$, we avoid using color 3 for the edges incident to $v_k$, so $v_k$ has at most $k-2$ edges of color 3. Consequently, the ordering $(x, v_1, \ldots, v_{k-2}, y)$ satisfies the assumptions of the lemma.

\medskip
\noindent
    {\bf Case 2.} $v_{k}$ has at most $k-3$ edges of color 3 in the coloring $c_1$.
    
     Let $c_2$ be an edge-coloring of $G_2$ such that 
    \begin{itemize}
        \item $c_2(xy)=2$;
        \item $c_2(xv_{i})=1$ for $i\in \{1,\ldots,{k-2}\}$;
        \item $c_2(xv_{i})=2$ for $i\in \{k+1,\ldots,2k-2\}$;
        \item $c_2(yv_{i})=3$ for $i\in \{1,\ldots,{k-2}\}$;
        \item $c_2(yv_{i})=4$ for $i\in \{k+1,\ldots,{2k-2}\}$;
        \item $c_2(xv_{k-1})=1$, $c_2(yv_{k-1})=4$;
        \item $c_2(xv_{k})=3$, $c_2(yv_{k})=3$.
    \end{itemize}

    Let $c$ be the edge-coloring of $K_{2k}$ defined by setting $c(e) = c_i(e)$ if $e \in E(G_i)$, for $i = 1, 2$. As in Case 1, we can show that $c$ is a majority edge-coloring. For each vertex $v_i$, where $i \in \{1, \ldots, k-2\}$, we have $\smc(v_i) = \sigma_{c_1}(v_i) + 4$, and for each vertex $v_i$ where $i \in \{k, \ldots, 2k-2\}$, we have $\smc(v_i) = \sigma_{c_1}(v_i) + 6$. Additionally, $\smc(v_{k-1}) = \sigma_{c_1}(v_{k-1}) + 5$. Thus, this coloring distinguishes the vertices $\{v_1, \ldots, v_{2k-2}\}$, with $\smc(v_i) < \smc(v_j)$ for $i < j$. Since $\smc(x) = 3k$ and $\smc(v_1) \ge 3(k-1) + 4$, we have $\smc(x) < \smc(v_1)$. Similarly, because $\smc(y) = 7k - 5$ and $\smc(v_{2k-2}) \ge 7(k-1)-5+6$, we find out that $\smc(y) > \smc(v_{2k})$. Therefore, $c$ is  an $\nsd$ edge-coloring. Finally, the ordering $(x, v_1, \ldots, v_{k-2}, y)$ satisfies the assumptions of the lemma, since $v_{k-1}$ is adjacent to at most $k-2$ edges of color 2.
     
\end{proof}

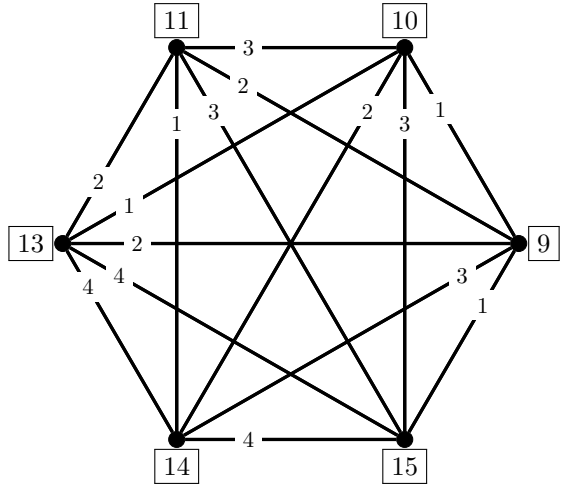
\begin{figure}[h]
\centering

\begin{tikzpicture}
    
    \def\n{6}
    \def\radius{3}

    \foreach \i in {1,...,\n} {
        
        \node[draw, vertex, fill=black, inner sep=1pt] (v\i) at ({360/\n * (\i - 1)}:\radius) {};
    }

\draw (v1) node[right] {\fbox{$9$}};
\draw (v2) node[above] {\fbox{$10$}};
\draw (v3) node[above] {\fbox{$11$}};
\draw (v4) node[left] {\fbox{$13$}};
\draw (v5) node[below] {\fbox{$14$}};
\draw (v6) node[below] {\fbox{$15$}};

    \foreach \i in {1,...,\n} {
        \foreach \j in {\i,...,\n} {
            \ifnum\i<\j
                \draw[edge] (v\i) -- (v\j);
            \fi
        }
    }

\draw[edge] (v4)to node[fill=white,pos=0.18 ,scale=0.85]{$1$} (v2);
\draw[edge] (v4)to node[fill=white, pos=0.3,scale=0.85]{$2$} (v3);
\draw[edge] (v4)to node[fill=white, pos=0.15,scale=0.85]{$2$} (v1);
\draw[edge] (v4)to node[fill=white, pos=0.2,scale=0.85]{$4$} (v5);
\draw[edge] (v4)to node[fill=white, pos=0.15,scale=0.85]{$4$} (v6);

\draw[edge] (v3)to node[fill=white,pos=0.3 ,scale=0.85]{$3$} (v2);
\draw[edge] (v3)to node[fill=white, pos=0.18,scale=0.85]{$1$} (v5);
\draw[edge] (v3)to node[fill=white, pos=0.15,scale=0.85]{$3$} (v6);
\draw[edge] (v3)to node[fill=white,pos=0.18 ,scale=0.85]{$2$} (v1);

\draw[edge] (v2)to node[fill=white, pos=0.15,scale=0.85]{$2$} (v5);
\draw[edge] (v2)to node[fill=white, pos=0.3,scale=0.85]{$1$} (v1);
\draw[edge] (v2)to node[fill=white,pos=0.18 ,scale=0.85]{$3$} (v6);

\draw[edge] (v1)to node[fill=white, pos=0.15,scale=0.85]{$3$} (v5);
\draw[edge] (v1)to node[fill=white, pos=0.3,scale=0.85]{$1$} (v6);

\draw[edge] (v5)to node[fill=white, pos=0.3,scale=0.85]{$4$} (v6);

\end{tikzpicture}
\caption{A majority $\nsd$ 4-edge-coloring of $K_{6}$.}
\label{fig:four}
\end{figure}

\begin{theorem}
 For $n\ge 3$, we have 
 $$\mnsdi\left(K_n \right) = \left\{ \begin{array}{ll}
3, & \textrm{ if $\quad n$ is odd},\\
4, & \textrm{ if $\quad n$ is even and $n\ge 6$,}\\
5, & \textrm{ if $n=4$.}\\
\end{array} \right.$$
\end{theorem}
\begin{proof}

As each vertex of $K_{2k+1}$ has even degree $2k$, it follows from Theorem~\ref{thm:complete} that
$\mnsdi(K_{2k+1})=\qmnsdi(K_{2k+1})=3$. For a majority $\nsd$ edge-coloring of the complete graph $K_{2k}$ of even order, at least four colors are required. This is because no vertex can have more than $k-1$ edges of a single color incident to it. Consequently, for colors from $[3]$, the smallest possible sum at a vertex is $k-1 + 2(k-1) + 3 = 3k$, and the largest possible sum is $3(k-1) + 2(k-1) + 1 = 5k-4$. Since we can only achieve $2k-3$ distinct sums and each vertex must have a unique sum, a minimum of four colors are necessary for a majority $\nsd$ edge-coloring. It is easy to verify that $K_4$ requires five colors.  However, by Lemma \ref{lem:majority complete} four colors suffice for a majority $\nsd$ edge-coloring of any complete graph with at least 6 vertices and of even order.
\end{proof}

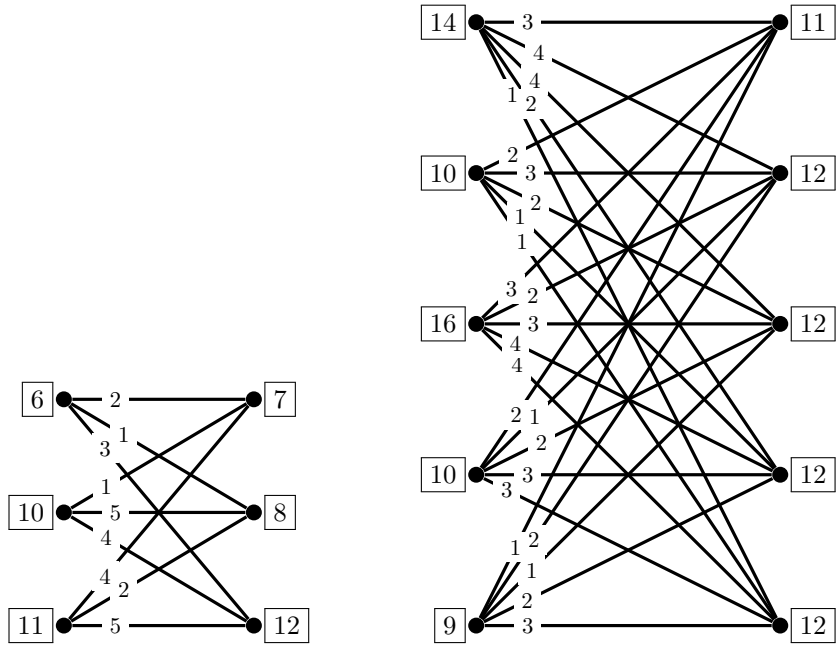
\begin{figure}[h]
\centering
\begin{tikzpicture}
\node[vertex] (1) at (0,0) {};
\node[vertex] (2) at (0,1.5) {};
\node[vertex] (3) at (0,3) {};
\node[vertex] (4) at (2.5,0) {};
\node[vertex] (5) at (2.5,1.5) {};
\node[vertex] (6) at (2.5,3) {}; 

\draw[edge] (1)to node[fill=white, near start,scale=0.85]{$5$} (4);
\draw[edge] (1)to node[fill=white, pos=0.3,scale=0.85]{$2$} (5);
\draw[edge] (1)to node[fill=white, pos=0.2,scale=0.85]{$4$} (6);
\draw[edge] (2)to node[fill=white, pos=0.2,scale=0.85]{$4$} (4);
\draw[edge] (2)to node[fill=white, near start,scale=0.85]{$5$} (5);
\draw[edge] (2)to node[fill=white, pos=0.2,scale=0.85]{$1$} (6);
\draw[edge] (3)to node[fill=white, pos=0.2,scale=0.85]{$3$} (4);
\draw[edge] (3)to node[fill=white, pos=0.3,scale=0.85]{$1$} (5);
\draw[edge] (3)to node[fill=white, near start,scale=0.85]{$2$} (6);

\draw[edge] (1) node[left] {\fbox{$11$}};
\draw[edge] (2) node[left] {\fbox{$10$}};
\draw[edge] (3) node[left] {\fbox{$6$}};
\draw[edge] (4) node[right] {\fbox{$12$}};
\draw[edge] (5) node[right] {\fbox{$8$}};
\draw[edge] (6) node[right] {\fbox{$7$}};
\end{tikzpicture}
\hspace{1cm}
   \begin{tikzpicture}
\node[vertex] (1) at (0,0) {};
\node[vertex] (2) at (0,2) {};
\node[vertex] (3) at (0,4) {};
\node[vertex] (4) at (0,6) {};
\node[vertex] (5) at (0,8) {};
\node[vertex] (6) at (4,0) {}; 
\node[vertex] (7) at (4,2) {};
\node[vertex] (8) at (4,4) {};
\node[vertex] (9) at (4,6) {};
\node[vertex] (10) at (4,8) {};

\draw[edge] (1)to node[fill=white,pos=0.15,scale=0.85]{$3$} (6);
\draw[edge] (1)to node[fill=white,pos=0.15,scale=0.85]{$2$} (7);
\draw[edge] (1)to node[fill=white,pos=0.17,scale=0.85]{$1$} (8);
 \draw[edge] (1)to node[fill=white,pos=0.18,scale=0.85]{$2$} (9);
\draw[edge] (1)to node[fill=white,pos=0.12,scale=0.85]{$1$} (10); 
    
\draw[edge] (2)to node[fill=white,pos=0.08,scale=0.85]{$3$} (6);
\draw[edge] (2)to node[fill=white,pos=0.15,scale=0.85]{$3$} (7);
\draw[edge] (2)to node[fill=white,pos=0.2,scale=0.85]{$2$} (8);
\draw[edge] (2)to node[fill=white,pos=0.18,scale=0.85]{$1$} (9);
\draw[edge] (2)to node[fill=white,pos=0.12,scale=0.85]{$2$} (10); 

\draw[edge] (3)to node[fill=white,pos=0.12,scale=0.85]{$4$} (6);
\draw[edge] (3)to node[fill=white,pos=0.11,scale=0.85]{$4$} (7);
\draw[edge] (3)to node[fill=white,pos=0.17,scale=0.85]{$3$} (8);
 \draw[edge] (3)to node[fill=white,pos=0.17,scale=0.85]{$2$} (9);
\draw[edge] (3)to node[fill=white,pos=0.1,scale=0.85]{$3$} (10);

\draw[edge] (4)to node[fill=white,pos=0.14,scale=0.85]{$1$} (6);
\draw[edge] (4)to node[fill=white,pos=0.13,scale=0.85]{$1$} (7);
\draw[edge] (4)to node[fill=white,pos=0.18,scale=0.85]{$2$} (8);
\draw[edge] (4)to node[fill=white,pos=0.16,scale=0.85]{$3$} (9);
\draw[edge] (4)to node[fill=white,pos=0.1,scale=0.85]{$2$} (10); 
    
\draw[edge] (5)to node[fill=white,pos=0.11,scale=0.85]{$1$} (6);
\draw[edge] (5)to node[fill=white,pos=0.17,scale=0.85]{$2$} (7);
\draw[edge] (5)to node[fill=white,pos=0.18,scale=0.85]{$4$} (8);
\draw[edge] (5)to node[fill=white,pos=0.19,scale=0.85]{$4$} (9);
\draw[edge] (5)to node[fill=white,pos=0.15,scale=0.85]{$3$} (10);

\draw[edge] (1) node[left] {\fbox{$9$}};
\draw[edge] (2) node[left] {\fbox{$10$}};
\draw[edge] (3) node[left] {\fbox{$16$}};
\draw[edge] (4) node[left] {\fbox{$10$}};
\draw[edge] (5) node[left] {\fbox{$14$}};
\draw[edge] (6) node[right] {\fbox{$12$}};
\draw[edge] (7) node[right] {\fbox{$12$}};
\draw[edge] (8) node[right] {\fbox{$12$}};
\draw[edge] (9) node[right] {\fbox{$12$}};
\draw[edge] (10) node[right] {\fbox{$11$}};
\end{tikzpicture}
\caption{A majority $\nsd$ 5-edge-coloring of $K_{3,3}$ and 4-edge-coloring of $K_{5,5}$.}
\label{fig:two}
\end{figure}

\begin{figure}[h]
\centering
\begin{tikzpicture}[scale=0.8]
\node[vertex] (1) at (0,0) {};
\node[vertex] (2) at (0,3) {};
\node[vertex] (3) at (0,6) {};
\node[vertex] (4) at (0,9) {};
\node[vertex] (5) at (0,12) {};
\node[vertex] (6) at (6,0) {}; 
\node[vertex] (7) at (6,3) {};
\node[vertex] (8) at (6,6) {};
\node[vertex] (9) at (6,9) {};
\node[vertex] (10) at (6,12) {};

\node[vertex] (a) at (0,15) {};
\node[vertex] (b) at (0,18) {};
\node[vertex] (c) at (6,15) {};
\node[vertex] (d) at (6,18) {};

\draw[edge] (1)to node[fill=white,pos=0.12,scale=0.75]{$1$} (6);
\draw[edge] (1)to node[fill=white,pos=0.12,scale=0.75]{$2$} (7);
\draw[edge] (1)to node[fill=white,pos=0.12,scale=0.75]{$3$} (8);
 \draw[edge] (1)to node[fill=white,pos=0.12,scale=0.75]{$3$} (9);
\draw[edge] (1)to node[fill=white,pos=0.12,scale=0.75]{$2$} (10); 
\draw[edge] (1)to node[fill=white,pos=0.12,scale=0.75]{$2$} (c); 
\draw[edge] (1)to node[fill=white,pos=0.12,scale=0.75]{$3$} (d); 
    
\draw[edge] (2)to node[fill=white,pos=0.103,scale=0.75]{$1$} (6);
\draw[edge] (2)to node[fill=white,pos=0.10,scale=0.75]{$2$} (7);
\draw[edge] (2)to node[fill=white,pos=0.108,scale=0.75]{$3$} (8);
\draw[edge] (2)to node[fill=white,pos=0.108,scale=0.75]{$3$} (9);
\draw[edge] (2)to node[fill=white,pos=0.108,scale=0.75]{$3$} (10); 
\draw[edge] (2)to node[fill=white,pos=0.108,scale=0.75]{$2$} (c); 
\draw[edge] (2)to node[fill=white,pos=0.108,scale=0.75]{$2$} (d); 

\draw[edge] (3)to node[fill=white,pos=0.098,scale=0.75]{$2$} (6);
\draw[edge] (3)to node[fill=white,pos=0.10,scale=0.75]{$1$} (7);
\draw[edge] (3)to node[fill=white,pos=0.10,scale=0.75]{$2$} (8);
 \draw[edge] (3)to node[fill=white,pos=0.10,scale=0.75]{$2$} (9);
\draw[edge] (3)to node[fill=white,pos=0.10,scale=0.75]{$3$} (10);
\draw[edge] (3)to node[fill=white,pos=0.10,scale=0.75]{$3$} (c); 
\draw[edge] (3)to node[fill=white,pos=0.105,scale=0.75]{$3$} (d); 

\draw[edge] (4)to node[fill=white,pos=0.10,scale=0.75]{$2$} (6);
\draw[edge] (4)to node[fill=white,pos=0.09,scale=0.75]{$1$} (7);
\draw[edge] (4)to node[fill=white,pos=0.09,scale=0.75]{$1$} (8);
\draw[edge] (4)to node[fill=white,pos=0.12,scale=0.75]{$1$} (9);
\draw[edge] (4)to node[fill=white,pos=0.13,scale=0.75]{$2$} (10); 
\draw[edge] (4)to node[fill=white,pos=0.09,scale=0.75]{$3$} (c); 
\draw[edge] (4)to node[fill=white,pos=0.1,scale=0.75]{$3$} (d); 
    
\draw[edge] (5)to node[fill=white,pos=0.108,scale=0.75]{$3$} (6);
\draw[edge] (5)to node[fill=white,pos=0.10,scale=0.75]{$3$} (7);
\draw[edge] (5)to node[fill=white,pos=0.10,scale=0.75]{$1$} (8);
\draw[edge] (5)to node[fill=white,pos=0.10,scale=0.75]{$1$} (9);
\draw[edge] (5)to node[fill=white,pos=0.10,scale=0.75]{$2$} (10); 
\draw[edge] (5)to node[fill=white,pos=0.10,scale=0.75]{$2$} (c); 
\draw[edge] (5)to node[fill=white,pos=0.10,scale=0.75]{$1$} (d); 

\draw[edge] (a)to node[fill=white,pos=0.108,scale=0.75]{$3$} (6);
\draw[edge] (a)to node[fill=white,pos=0.105,scale=0.75]{$3$} (7);
\draw[edge] (a)to node[fill=white,pos=0.10,scale=0.75]{$2$} (8);
\draw[edge] (a)to node[fill=white,pos=0.10,scale=0.75]{$2$} (9);
\draw[edge] (a)to node[fill=white,pos=0.10,scale=0.75]{$1$} (10); 
\draw[edge] (a)to node[fill=white,pos=0.10,scale=0.75]{$1$} (c); 
\draw[edge] (a)to node[fill=white,pos=0.10,scale=0.75]{$1$} (d); 

\draw[edge] (b)to node[fill=white,pos=0.116,scale=0.75]{$3$} (6);
\draw[edge] (b)to node[fill=white,pos=0.115,scale=0.75]{$3$} (7);
\draw[edge] (b)to node[fill=white,pos=0.11,scale=0.75]{$2$} (8);
\draw[edge] (b)to node[fill=white,pos=0.11,scale=0.75]{$2$} (9);
\draw[edge] (b)to node[fill=white,pos=0.11,scale=0.75]{$1$} (10); 
\draw[edge] (b)to node[fill=white,pos=0.11,scale=0.75]{$1$} (c); 
\draw[edge] (b)to node[fill=white,pos=0.11,scale=0.75]{$1$} (d); 
    
\draw[edge] (1) node[left] {\fbox{$16$}};
\draw[edge] (2) node[left] {\fbox{$16$}};
\draw[edge] (3) node[left] {\fbox{$16$}};
\draw[edge] (4) node[left] {\fbox{$13$}};
\draw[edge] (5) node[left] {\fbox{$13$}};
\draw[edge] (6) node[right] {\fbox{$15$}};
\draw[edge] (7) node[right] {\fbox{$15$}};
\draw[edge] (8) node[right] {\fbox{$14$}};
\draw[edge] (9) node[right] {\fbox{$14$}};
\draw[edge] (10) node[right] {\fbox{$14$}};
\draw[edge] (a) node[left] {\fbox{$13$}};
\draw[edge] (b) node[left] {\fbox{$13$}};
\draw[edge] (c) node[right] {\fbox{$14$}};
\draw[edge] (d) node[right] {\fbox{$14$}};
\end{tikzpicture}
\caption{A majority $\nsd$  3-edge-coloring of $K_{7,7}$.}
\label{fig:three}
\end{figure}
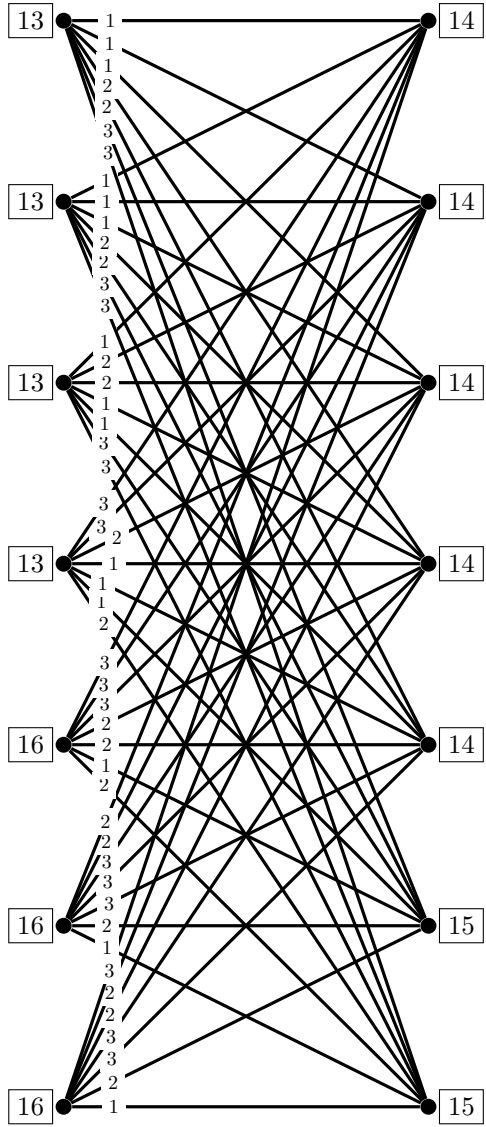

\begin{theorem}
 Let both $m,n$ be even. Then 
 $$\mnsdi\left(K_{n,m} \right) = \left\{ \begin{array}{ll}
4, & \textrm{ if $\; n=m=2$},\\
3, & \textrm{ if $\; n=m\geq 4$},\\
2, & \textrm{ if $\; n\neq m$}. 
\end{array} \right.$$
If at least one of the integers $n,m$ is odd and $n,m \ge 2$, then 
$$\mnsdi\left(K_{n,m} \right) = \left\{ \begin{array}{ll}
5, & \textrm{ if $\; n=m=3$},\\
4, & \textrm{ if $\; n=m= 5$},\\
3, & \textrm{ otherwise. }
\end{array} \right.$$
\end{theorem}

\begin{proof}
Theorem~\ref{compbip} implies that when both $n$ and $m$ are even and $K_{n,m} \neq K_{2,2}$, we have $\mnsdi\left(K_{n,m}\right) = 2$ if $n \neq m$, and $\mnsdi\left(K_{n,n}\right) = 3$ for $n\ge 4$. Otherwise, if $K_{n,m}$ contains vertices of odd degrees at least 3, that is, for odd $m$ or $n$, then at least three colors are necessary for a majority neighbor sum distinguishing edge-coloring. However, in the case of $K_{3,3}$, even four colors are insufficient to distinguish adjacent vertices. 
To see this, observe that in any  majority $\nsd$ 4-edge-coloring  $c$ of $K_{3,3}$, we have $\smc(v) \in \{6, 7, 8, 9\}$ for every $v \in V(K_{3,3})$.  If $\smc(v) = 9$, then there is no edge colored with 1 incident to $v$. On the other hand, if $\sigma_c(v) \in \{6, 7, 8\}$, then $v$ must be incident to an edge colored with 1.
Thus, to achieve a majority $\nsd$ edge-coloring, no vertex can have the sum of 9. However, it can be easily verified that if $\smc(v) \in \{6, 7, 8\}$ for every vertex $v \in V(K_{3,3})$, it becomes impossible to distinguish between the vertices of different sets in the bipartition.
For $K_{5,5}$, we need at least four colors for a majority $\nsd$ edge-coloring. To see that there is no majority $\nsd$ 3-edge-coloring of $K_{5,5}$, observe that in any such coloring $c$, we have $\smc(v) \in \{9, 10, 11\}$ for every $v \in V(K_{5,5})$. This makes it impossible to distinguish between vertices from the two different sets of the bipartition.  Figure \ref{fig:two} presents a majority $\nsd$ 5-edge-coloring of $K_{3,3}$ and a majority $\nsd$ 4-edge-coloring of $K_{5,5}$. Thus, we have $\mnsdi(K_{3,3}) = 5$ and $\mnsdi(K_{5,5}) = 4$.

Now, consider the remaining cases, that is, when at least one of the integers $n,m$ is odd. 

Let $K_{n,m}=(V_1,V_2,E)$, where $V_1=\left\{a_1,\ldots,a_n\right\}$, $V_2=\left\{b_1,\ldots,b_m\right\}$.  Let $c':E(K_{n,m})\rightarrow [n+m-1]$ be the interval coloring of $K_{n,m}$ such that $c'(a_ib_j)=i+j-1$. First, we consider all cases except when $n = m$. In these cases, using the coloring $c'$, we construct a new edge-coloring $c: E(K_{n,m}) \rightarrow [3]$ by setting $c(e) \equiv c'(e) \pmod{3}$. It is straightforward to see that $c$ is a majority coloring, as $\lceil d/3 \rceil \leq \lfloor d/2 \rfloor$ for $d \geq 2$. Now, we claim that $c$ is also distinguishing. 

If $n=3k$ and $m=3\ell$, then $\smc(a_i)=6\ell$  and $\smc(b_j)=6k$ for $i \in [n]$ and $j \in [m]$. By our assumption, $k\neq \ell$, so $\smc(a_i) \neq \smc(b_j)$ for all $i \in [n]$ and $j \in [m]$, which means that $c$ is a majority $\nsd$ edge-coloring.
If $n=3k$ and $m=3\ell+1$ (or $m=3\ell+2$), then $\smc(a_i) \in \{6\ell+1, 6\ell+2, 6\ell+3\}$ (or $\smc(a_i) \in \{6\ell+3, 6\ell+4,  6\ell+5\}$) and $\smc(b_j)=6k$ for $i \in [n]$ and $j \in [m]$. Therefore, $\smc(a_i) \neq \smc(b_j)$ for all $i \in [n]$ and $j \in [m]$. Thus, $c$ is a majority $\nsd$ edge-coloring. 
If $n=3k+1$ and $m=3\ell+1$, then $\smc(a_i)\in \{6\ell+1, 6\ell+2,6\ell+3\}$  and $\smc(b_j)\in \{6k+1,6k+2,6k+3\}$ for $i \in [n]$ and $j \in [m]$. Since $k\neq \ell$, we have $\smc(a_i) \neq \smc(b_j)$ for all $i \in [n]$ and $j \in [m]$. Similarly, if $n=3k+2$ and $m=3\ell+2$, then $\smc(a_i)\in \{6\ell+3, 6\ell+4,6\ell+5\}$  and $\smc(b_j)\in \{6k+3,6k+4,6k+5\}$ for $i \in [n]$ and $j \in [m]$. As in the previous case, since $k\neq \ell$,  we have $\smc(a_i) \neq \smc(b_j)$ for all $i \in [n]$ and $j \in [m]$.

Suppose $n = 3k + 1$ and $m = 3\ell + 2$. In this case, we have $\smc(a_i) \in \{6\ell + 3, 6\ell + 4, 6\ell + 5\}$ and $\smc(b_j) \in \{6k + 1, 6k + 2, 6k + 3\}$ for $i \in [n]$ and $j \in [m]$. If $k \neq \ell$, then $\smc(a_i) \neq \smc(b_j)$ for all $i \in [n]$ and $j \in [m]$, resulting in a majority $\nsd$ edge-coloring.
However, if $k = \ell$ (i.e., for $K_{3k+1, 3k+2}$), we need to recolor some edges. Specifically, we recolor the edges $a_i b_i$ for $i \in \{1, 4, \ldots, 3k + 1\}$, where initially $c(a_i b_i) = 1$. We recolor these edges to 3. After recoloring, we have $\smc(a_i) = 6k + 5$ and $\smc(b_j) = 6k + 3$ for $i, j\in \{1, 4, \ldots, 3k + 1\}$.
Thus, after recoloring, $\smc(a_i) \in \{6k + 4, 6k + 5\}$ and $\smc(b_j) \in \{6k + 2, 6k + 3\}$ for all $i \in [3k+1],j\in [3k+2]$, ensuring a  $\nsd$ edge-coloring. Furthermore, for $i \in \{1, 4, \ldots, 3k + 1\}$, vertex $a_i$ has $k$ edges in color 1, $k + 1$ edges in color 2, and $k + 1$ edges in color 3. Similarly, for $j \in \{1, 4, \ldots, 3k + 1\}$, vertex $b_j$ has $k$ edges in color 1,  $k$ edges in color 2, and $k + 1$ edges in color 3. Therefore, the coloring remains a majority edge-coloring.

Finally, suppose $n = m = 2k + 1$.  Figure \ref{fig:three} presents a majority $\nsd$ 3-edge-coloring of $K_{7,7}$. Thus, we may assume that $k\geq 4$. In this case, using the coloring $c'$, we first construct a new edge-coloring $c: E(K_{n,m}) \rightarrow [2]$ by setting $c(e) \equiv c'(e) \pmod{2}$. As a result, we have $\smc(a_i) = \smc(b_i) = 3k + 1$ for $i \in \{1, 3, \ldots, 2k + 1\}$ and $\smc(a_i) = \smc(b_i) = 3k + 2$ for $i \in \{2, 4, \ldots, 2k\}$.
In the next step, we recolor certain edges with color 3. Specifically, we assign:
\begin{enumerate}
    \item $c(a_ib_i)=3$, $c(a_ib_{i+1})=3$, $c(a_{i+1}b_i)=3$, $c(a_{i+1}b_{i+1})=3$ for  $i\in \{1,3,5,\ldots, 2k-1\}$,
    \item $c(a_{2k+1}b_{2k+1})=3$,
    \item $c(a_ib_{i+1})=3$, $c(a_ib_{i+2})=3$ for $i\in \{2,4,\ldots, 2k-2\}$,
    \item $c(a_{2k}b_{1})=3$, $c(a_{2k}b_{2})=3$,
    \item $c(a_1b_{2k+1})=3$.
\end{enumerate}
After steps 1 and 2 we have $\smc(a_i) =\smc(b_i)= 3k + 4$ for $i \in \{1, 3, \ldots, 2k - 1\}$ and $\smc(a_i)=\smc(b_i) = 3k + 5$ for $i \in \{2, 4, \ldots, 2k\}$ and $\smc(a_{2k+1})=\smc(b_{2k+1}) = 3k + 3$.

After steps 3, 4, 5 of recoloring, as a result, we have:
\begin{itemize}
    \item $\smc(a_i)=3k + 8$ for $i \in \{2,4,\ldots,2k\}$,
    \item $\smc(a_i)= 3k + 4$ for $i \in \{3,5,7,\ldots 2k-1\}$,
    \item $\smc(a_1) = 3k +6$, $\smc(a_{2k+1}) = 3k +3$,
     \item $\smc(b_i)=3k + 7$ for $i \in \{2,4,\ldots,2k\}$,
    \item $\smc(bi)= 3k + 5$ for $i \in \{1,3,\ldots 2k+1\}$.
\end{itemize}

Thus, $c$ is an $\nsd$ edge-coloring. Since at each vertex (except $a_{2k+1}$) at least one edge was recolored from 1 to 3 and at least one edge from 2 to 3, each vertex (except $a_{2k+1}$) has at most $k$ edges in color 1 and at most $k$ edges in color 2. At vertex $a_{2k+1}$, the edge previously colored 1 was recolored to 3. Therefore, even at $a_{2k+1}$, there are at most $k$ edges colored 1 and at most $k$ edges colored 2. Additionally, there are at most four edges in color 3 at each vertex. Hence, the coloring is a majority $\nsd$ edge-coloring.
\end{proof}

\bigskip
The following result provides a general upper bound for $\mnsdi(G)$. Although its proof is analogous to that of Theorem~\ref{genqm}, we provide it for convenience. 
\begin{theorem}\label{genm}
For every graph $G$ with $\delta(G)\ge 2$, $$\mnsdi(G)\leq 18.$$    
\end{theorem}
\begin{proof}
By Theorem~\ref{m4}, $G$ has a majority 4-edge coloring.
Again, we may assume that $G$ is connected. We will construct a majority NSD coloring $c:E(G)\rightarrow [18]$.

Order the vertex set $V(G)=\{v_1,\ldots, v_n\}$ such that each vertex $v_i$, for $i<n$, has a~neighbor $v_j$ with $j>i$. Let $c_0$ be a majority $4$-edge-coloring of $G$ with colors $7,8,9,10$. Set an initial value $c(e)=c_0(e)$ for each edge $e\in E(G)$.  To every vertex $v_i$ with $i\leq n-1$, we will assign a~set $W(v_i)=\{w(v_i),w(v_i)+6\}$ of possible final values of $\sigma_c(v_i)$ such that $w(v_i)\in\{0,\ldots 5\} \pmod {12}$, and $W(v_i)\cap W(v_j)= \emptyset$ for every $v_iv_j\in E(G)$. 

In the first step, we count $\smc(v_1)=\sum_{u \in N(v_1)}c_0(v_1u)$, and define the sets $W(v_1)=\{w(v_1),w(v_1)+6\}$ by putting $w(v_1)=\smc(v_1)$ if $\smc(v_1)\in \{0,\ldots,5\}\pmod {12}$, or $w(v_1)=\smc(v_1)-6$ otherwise. 

Let $2\leq k\leq n-1$, and assume that we have already established the set $W(v_i)$ for each $i<k$, and

\noindent
$(1)\quad c(v_iv_j)\in [18]$ for $1\leq i<j\leq n$,\\
$(2) \quad\smc(v_i)\in W(v_i)$ for $i<k$, \\
$(3)\quad c(v_kv_j)=c_0(v_jv_k)$ for $j>k$,\\
$(4)\quad  c(v_iv_k)\in \{7,\ldots 12\}$ for $i<k$. 

In view of condition $(4)$, if $v_iv_k\in E(G)$, then we can either add or subtract $6$ to $c(v_iv_k)$, so that the resulting $c$ is still a majority coloring ,and $\smc(v_i)\in W(v_i)$. If $v_k$ has $d$ neighbors $v_i$ with $i<k$, then this gives us $d+1$ choices for $\smc(v_k)$. Furthermore, we can change $c(v_kv_{j_0})$, where $j_0$ is the smallest $j>k$ with $v_kv_j\in E(G)$, using the following rule. It is not difficult to see that if we want to change the color of $c(v_kv_{j_0})$, then for each of its end-vertices $v_k,v_{j_0}$, there are at most two colors violating the majority coloring of that vertex. Hence, we can choose a~new color $c(v_kv_{j_0})\in \{7,\ldots,12\}$ such that $c$ is still a~majority edge-coloring of $G$. This gives us additional $d$ choices for $\smc(v_k)$. In total, we have $2d+1$ possible values for $\smc(v_k)$, at least one of them does not belong to any $W(v_i)$ for $i<k$. Therefore, a~possible recoloring of some edges $v_iv_k$ with $i<k$ and an edge $v_kv_{j_0}$ results in an edge-coloring $c$ satisfying the conditions:

\noindent
$(1')\quad c(v_iv_j)\in [18]$ for $1\leq i<j\leq n$,\\
$(2') \quad\smc(v_i)\in W(v_i)$ for $i\leq k$, \\
$(3')\quad c(v_kv_j)=c_0(v_jv_k)$ for $j>k$, except for $j=j_0$,\\
$(4')\quad  c(v_iv_n)\in \{7,\ldots,12\}$ for $i<n$. 

This way, we successively assign pairwise disjoint sets $W(v_k)$ to all $k\leq n-1$.

In the final step, we have to determine $\smc(v_n)$. If $v_iv_n\in E(G)$ for some $i<n$, then condition $(4')$ allows us to subtract or add 6 to $c(v_iv_n)$ such that $\smc(v_i)\in W(v_i)$. Hence, we have $d(v_n)+1\geq 3$ possible options for $\smc(v_n)$. Let $s$ be the smallest such possible option, which we obtain by subtracting 6 from $c(v_i)$ for all $v_i\in N(v_n)$ with $\smc(v_i)=w(v_i)+6$, and adding 6 nowhere. If $s\in \{6,\ldots,11\} \pmod {12}$, then $s$ cannot be equal to $\smc(v_i)$ for any neighbor $v_i$ of $v_n$ since otherwise we could increase $s$ by subtracting 6 from $c(v_iv_n)$. Hence, we can put $\smc(v_n)=s$. Let then $s\in\{0,\ldots,5\} \pmod {12}$. If there exists a~$v_i\in N(v_n)$ with $\smc(v_i)\neq s$, then we keep $c(v_i)$ unchanged and subtract 6 everywhere else, thus obtaining a~suitable value $\smc(v_n)=s+6$. If $\smc(v_i)=w(v_i)$ for all $v_i\in N(v_n)$ then we subtract 6 from all edges incident to $v_n$ except two of them. 

The resulting values of $\smc$ yield a~proper vertex-coloring of $G$.     
\end{proof}

It is not difficult to modify the proof of Theorem~\ref{genm} to justify the following two results.
\begin{proposition}\label{delta4}
If $G$ is a graph with minimum degree at least $4$, then $\mnsdi(G)\leq 15$.   
\end{proposition}
\begin{proposition}\label{meven}
If $G$ is a graph of even size with all vertices of even degrees, then $\mnsdi(G)\leq 12$.
\end{proposition}
Indeed, Proposition \ref{delta4} follows from the fact proven in~\cite{BKPPRW} that every graph $G$ with $\delta(G)\ge 4$ admits a majority 3-edge-coloring. To prove Proposition~\ref{meven}, we apply item 1. of Theorem~\ref{T2}, which implies that a graph of even size with all vertices of even degrees has a majority 2-edge-coloring.


\section{Conclusion and open problems}\label{sec:open problems}

In this paper, we explore the quasi-majority neighbor sum distinguishing edge-coloring of graphs. We provide general upper bounds and determine exact values or upper bounds for the $\qmnsd$ index of specific graph classes.

A key question arising from our results is to determine the minimum integer $k$ such that every graph admits a quasi-majority neighbor sum distinguishing $k$-edge-coloring. We prove that 12 colors suffice for any graph. However, our findings for specific graph classes suggest that this number could be much smaller. The graph with the highest known $\qmnsd$ index 5 is $C_5$. Thus, the primary open problem is to identify an infinite family of graphs with a $\qmnsd$ index of at least 5. Nevertheless, we believe that such a family might not exist and propose the following conjecture.
\begin{conjecture}
Every nice graph $G \neq C_5$ satisfies $\qmnsdi(G) \leq 4$.
\end{conjecture}
For majority colorings, we suppose the following.
\begin{conjecture}
Every graph $G$ with $\delta(G)\ge 2$ satisfies $\mnsdi(G) \leq 5$.
\end{conjecture}


\end{document}